\documentclass[12pt]{article}
\usepackage[latin1]{inputenc}
\usepackage[T1]{fontenc}
\usepackage{amsmath,amssymb,amstext}
\usepackage{theorem}
\usepackage{multicol}
\usepackage[english]{babel}
\usepackage{enumerate}
\usepackage{eufrak}

\newtheorem{theorem}{Theorem}

\newtheorem{lemma}[theorem]{Lemma}

\newtheorem{proposition}[theorem]{Proposition}
\newtheorem{remark}[theorem]{Remark}

\newenvironment{proof}[1][Proof]{\textbf{#1.} }{\ \rule{0.5em}{0.5em}}

\renewcommand{\geq}{\geqslant}

\def \Hf    {\mathfrak{H}}

\def\1{{\mathbf{1}}}

\def\1{{\mathbf{1}}}
\def\0.5{{\frac{1}{2}}}

\renewcommand{\thefootnote}{\fnsymbol{footnote}}

\DeclareMathOperator{\e}{e}

\def \Hf    {\mathcal{H}}

\begin{document}

\qquad \renewcommand{\thefootnote}{\arabic{footnote}}

\begin{center}
{\Large \textbf{Berry-Esséen bounds for parameter estimation of general
Gaussian processes }} \\[0pt]
~\\[0pt]
Soukaina Douissi\footnote{%
Cadi Ayyad University, Marrakesh, Morocco. Email: \texttt{%
douissi.soukaina@gmail.com}}, Khalifa Es-Sebaiy\footnote{%
National School of Applied Sciences - Marrakesh, Cadi Ayyad University,
Marrakesh, Morocco. Email: \texttt{k.essebaiy@uca.ma}} and Frederi G. Viens
\footnote{%
Department of Statistics and Probability, Michigan State University 619 Red
Cedar Rd East Lansing, MI 48824, USA. E-mail: \texttt{viens@msu.edu}}\\[0pt]
\textit{Cadi Ayyad University and Michigan State University}\\[0pt]
~\\[0pt]
\end{center}

{\ \noindent \textbf{Abstract:} We study rates of convergence in central
limit theorems for the partial sum of squares of general Gaussian sequences,
using tools from analysis on Wiener space. No assumption of stationarity,
asymptotically or otherwise, is made. The main theoretical tool is the
so-called Optimal Fourth Moment Theorem \cite{NP2015}, which provides a
sharp quantitative estimate of the total variation distance on Wiener chaos
to the normal law. The only assumptions made on the sequence are the
existence of an asymptotic variance, that a least-squares-type estimator for
this variance parameter has a bias and a variance which can be controlled,
and that the sequence's auto-correlation function, which may exhibit long
memory, has a no-worse memory than that of fractional Brownian motion with
Hurst parameter }$H<3/4$.{\ \ Our main result is explicit, exhibiting the
trade-off between bias, variance, and memory. We apply our result to study
drift parameter estimation problems for subfractional Ornstein-Uhlenbeck and
bifractional Ornstein-Uhlenbeck processes with fixed-time-step observations.
These are processes which fail to be stationary or self-similar, but for
which detailed calculations result in explicit formulas for the estimators'
asymptotic normality.\vspace*{0.1in}}

{\noindent }\textbf{Key words}: Central limit theorem; Berry-Esseen;
stationary Gaussian process; Nourdin-Peccati analysis; parameter estimation;
fractional Brownian motion; long memory.\vspace*{0.1in}


\section{Introduction}

This paper presents the estimation of the asymptotic variance, when it
exists, of a general Gaussian sequence, and applies it to non-stationary
stochastic processes based on fractional noises, which can have long memory.
We choose to work with a simple method based on empirical second moments,
which is akin to the discretization of a least-squares method for the
corresponding processes using discrete observations with no in-fill
assumption (fixed time step).

Our work responds to the recent preprint \cite{EV}, which was the first
instance in which this type of estimation was performed without in-fill
assumptions on the data, but which insisted on keeping the data stationary
or asymptotically so. That paper \cite{EV}, available on arXiv, contains an
extensive description of the literature on parameter estimation for Gaussian
stochastic processes. We do not repeat that analysis here. Instead, we list
a number of recent articles which deal with various versions of
least-squares-type parameter estimation for fractional-noise-driven
Ornstein-Uhlenbeck processes (fOU-type), since that class of processes is
similar to the various examples we consider in this article: \cite{HN}, \cite%
{AM}, \cite{AV}, \cite{EEV}, \cite{HS}, \cite{BI}, \cite{NT}, \cite{BEO},
\cite{EEO} and \cite{EET}. respectively. Those papers are described in the
context of our motivations in \cite{EV}. We now highlight the distinction
between our paper and \cite{EV}.

In that work, the assumption of stationarity of the data was weakened by
asking for asymptotic stationarity, where the deviation from a stationary
model converged to 0 exponentially fast. This strong assumption only allowed
to work with specific examples which are close to stationary. In that
article, Ornstein-Uhlenbeck models were covered, as long as the driving
noise was stationary, as it would be for standard fractional Brownian
motion. That paper was the first instance in which parameters for fOU-type
processes could be estimated with asymptotic normality using only discrete
data with fixed time step. However, this range of application in \cite{EV}
extended to non-stationary processes which benefitted from an exponential
ergodicity property. One leaves this range of applicabilith when driving OU
equations with noises which are not stationary to begin with. This is our
primary motivation in this article: to develop a framework in which no
assumption of stationarity is made at the level of driving noise processes.
Instead, we investigate a set of arguably minimal assumptions on how to
estimate an asymptotic variance, if it exists, for any discretely observed
Gaussian process. In this article, we make no a-priori assumption about the
speed of convergence of the data to a stationary model, only asking that
convergence occur.

This implies that the main result of the paper, our estimator's asymptotic
normality, applies to some of the more exotic non-stationary covariance
structures, for which the estimator properties cannot be handled by any
classical or current method. Our main result isolates three terms in order
to construct estimates which are fully quantitative. We show that the
Wasserstein distance between the CLT scaling of our discrete-observation
least-squares estimator of the asymptotic variance and a normal law is
bounded above by the sum of a bias term, a variance term, and a term which
comes from an application of the optimal fourth moment theorem of Nourdin
and Peccati \cite{NP2015}, and depends only on an upper bound on the
covariance structure of the discrete observations. This allows us to
formulate clear assumptions on what is needed to turn this estimate into a
quantiative asymptotic normality theorem for the estimator, in Wasserstein
distance. The advantage of working in this fashion is that one has an
immediate way of identifying possible mean-variance-correlation trade-offs,
since the corresponding three terms are explict and are simply added
together as a measure of how fast the estimator converges to a normal. There
are additional features to our method which are described in Section 3, such
as our ability to separate our minimal convergence assumptions between what
is needed for strong consistency and what is needed in addition to obtain
asymptotic normality: see Hypotheses $(\mathcal{H}1)$ thru $(\mathcal{H}4)$
in Section 3, the bullet points that follow, and the statements of Theorems %
\ref{thm cv in law of Fn} and \ref{berry ess. theta gen. case} in the same
section.

When applying our general method from Section 3 to specific examples in
Section 4, we consider the solutions of the OU equation driven by
sub-fractional Brownian motion and by bi-fractional Brownian motion, both of
which have somewhat non-trivial covariance structures with non-stationary
increments, which are then inherited by the corresponding fOU-type processes
whose drift parameter $\theta $ is of interest. These fOU-type processes
have been well studied in the literature over the last decade, including
efforts to understand their statistical inference. We refer only to \cite%
{EEO} for information about these processes and additional references. We
compute the asymptotic variances of these fOU-type processes, which are
explicit functions $f\left( \theta \right) $, depending only on $\theta $
(and on the Hurst-type parameters which are assumed to be known), and are in
fact power functions of $\theta $. The normal asymptotics for the estimators
of each of these two $f\left( \theta \right) $'s can be converted to similar
results for $\theta $ itself; see \cite{EV} for instance for an analysis,
which we do not explore further here. Part of being able to apply the
strategy of Section 3 successfully in the examples of Section 4, relies on
three things, the first two which cannot be avoided, the third which enables
an efficient estimation of convergence speeds:

\begin{enumerate}
\item being able to identify $f\left( \theta \right) $, preferably in an
explicit way, thus giving access to $\theta $, but at least to show that the
data stream's variance converges;

\begin{enumerate}
\item we do this for both the sub-fractional and bi-fractional OU processes;

\item the expressions for $f\left( \theta \right) $ are simple (see the
formulas for $f_{H}\left( \theta \right) $ and $f_{H,K}\left( \theta \right)
$ for each process respectively, in (\ref{fH}) and (\ref{fHK})) but proving
convergence and evaluating the speed therein is slightly non-trivial;
\end{enumerate}

\item showing that the CLT-normalized estimator of $f\left( \theta \right) $
has an asymptotic variance $\sigma ^{2}$; this does not need to be found
explicitly, as long as one can prove that the variances converge;

\begin{enumerate}
\item we succeed in this task for both the sub-fractional and bi-fractional
OU processes, and are moreover able to compute $\sigma ^{2}$
semi-explicitly, using a series representation based on the covariance
function of discrete-time fractional Gaussian noise (fGn); see formulas (\ref%
{sigmaH}) and (\ref{sigmaHK}) respectively;

\item these expressions comes as something of a surprise, since \emph{a
priori} there was no reason to expect that the estimator's asymptotic
variance could be expressed using the fGn's covariance rather than the
non-stationary covariance structure of the increments of the sub- or
bi-fractional processes;
\end{enumerate}

\item computing speeds of convergence of both the estimator's bias and its
asymptotic variance, as explicitly as possible;

\begin{enumerate}
\item this is done for both the sub-fractional and bi-fractional OU
processes, in a rather explicit fashion, and requires a certain amount of
elbow grease, where most calculations are relegated to lemmas in the
Appendix;

\item we do not claim to have performed these estimations as efficiently as
can possibly be done; however, we think that the power orders of these
convergences are probably optimal because: (i) we have based the estimations
on the optimal fourth moment theorem of Nourdin and Peccati \cite{NP2015},
(ii) we have confirmed a phenomenon observed in a simpler (stationary)
context in \cite{NV2014} (third moment theorem) by which the speed of
convergence is given by that of an absolute third moment rather than by a
fourth cumulant, and (iii) our final result for the sub- and bi-fractional
processes identify a well-known threshold of a Hurst parameter not exceeding
$3/4$ as the upper endpoint of the range where quadratic variations are
asymptotically normal.
\end{enumerate}
\end{enumerate}

To finish this introduction, we note the general structure of this paper.
The next section presents preliminaries regarding the analysis of Wiener
space, and other convenient facts, which are needed in some technical parts
of the paper. The reader interested in our results more than our proofs can
safely skip this Section 2 upon a first reading. Section 3 presents the
general framework of how to estimate the asymptotic variance of a Gaussian
process observed in discrete time. Section 4 shows how this method can be
implemented in two cases of discretely observed OU processes driven by two
well-known non-stationary Gaussian processes with medium-to-long memory. The
Appendix collects calculations useful in Section 4.

\section{Preliminaries, and elements of analysis on Wiener space\label%
{Wiener}}

This section provides essential facts from basic probability and the
analysis on Wiener space and the Malliavin calculus. These facts and their
corresponding notation underlie all the results of this paper, but most of
our results and arguments can be understood independently of this section.
The interested reader can find more details in \cite[Chapter 1]{nualart-book}
and \cite[Chapter 2]{NP-book}.

\subsection{A convenient lemma\label{CLemma}}

The following result is a direct consequence of the Borel-Cantelli Lemma
(the proof is elementary; see e.g. \cite{KN}). It is convenient for
establishing almost-sure convergences from $L^{p}$ convergences.

\begin{lemma}
\label{Borel-Cantelli} Let $\gamma >0$. Let $(Z_{n})_{n\in \mathbb{N}}$ be a
sequence of random variables. If for every $p\geq 1$ there exists a constant
$c_{p}>0$ such that for all $n\in \mathbb{N}$,
\begin{equation*}
\Vert Z_{n}\Vert _{L^{p}(\Omega )}\leqslant c_{p}\cdot n^{-\gamma },
\end{equation*}%
then for all $\varepsilon >0$ there exists a random variable $\eta
_{\varepsilon }$ which is almost such that
\begin{equation*}
|Z_{n}|\leqslant \eta _{\varepsilon }\cdot n^{-\gamma +\varepsilon }\quad %
\mbox{almost surely}
\end{equation*}%
for all $n\in \mathbb{N}$. Moreover, $E|\eta _{\varepsilon }|^{p}<\infty $
for all $p\geq 1$.
\end{lemma}

A typical application is when a sequence in a finite sum of Wiener chaoses
converges in the mean square. By the equivalence of all $L^{p}$ norms in
Wiener chaos, which is a consequence of the hypercontractivity of the
Ornstein-Uhlenbeck semigroup on Wiener space (see \cite[Section 2.8]{NP-book}%
, and Section \ref{Elements} below), the lemma's assumption is then
automatically satisfied if the speed of convergence in the mean square is of
the same power form.

The random variable $\eta _{\varepsilon }$ can typically be chosen more
explicitly than what the lemma guarantees. For instance, it is well known
that if $Z$ is Gaussian, then $\eta _{\varepsilon }$ can be chosen as a
\emph{gauge} function of $Z$, and is therefore sub-Gaussian (tails no larger
than Gaussian; see \cite{F1975}). A similar property holds for processes in
a finite sum of Wiener chaoses, as can be inferred from \cite{VV2007}. We
will not pursue these refinement, since the above lemma will be sufficient
for our purposes herein.

\subsection{The Young integral}

Fix $T>0$. Let $f,g:[0,T]\longrightarrow \mathbb{R}$ be Hölder continuous
functions of orders $\alpha \in (0,1)$ and $\beta \in (0,1)$ respectively
with $\alpha +\beta >1$. Young \cite{young} proved that the
Riemann-Stieltjes integrals (so-called Young integrals) $%
\int_{0}^{.}f_{u}dg_{u}$ and $\int_{0}^{.}g_{u}df_{u}$ exist as limits of
the usual discretization. Morever, for all $t\in \lbrack 0,T]$, integration
by parts (product rule) holds:
\begin{equation}
f_{t}g_{t}=f_{0}g_{0}+\int_{0}^{t}g_{u}df_{u}+\int_{0}^{t}f_{u}dg_{u}.
\label{Young}
\end{equation}

This integral is convenient to define stochastic integrals in a pathwise
sense with respect to processes $f$ which are smoother than Brownian motion,
in the sense that they are almost surely $\alpha $-Hölder continuous for
some $\alpha >1/2$. A typical example of such a process is the fractional
Brownian motion with Hurst parameter $H>1/2$. In that case, any $\alpha $
can be chosen in $\left( 1/2,H\right) $, and $g$ can be any process with the
same regularity as the fBm, and thus with $\beta =\alpha $; this enables a
stochastic calculus immediately, in which no \textquotedblleft Itô%
\textquotedblright -type correction terms occur, owing to (\ref{Young}). For
instance, for any Lipshitz non-random function $g$, this
integration-by-parts formula holds for any a.s. Hölder-continuous Gaussian
process $f$, since then $\alpha >0$ and $\beta =1$. In particular, the first
integral in (\ref{Young}) coincides with the Wiener integral, the second is
an ordinary Riemann-Stieltjes integral, and no Itô correction occurs since
one of the processes is of finite variation.

\subsection{Elements of Analysis on Wiener space\label{Elements}}

With $\left( \Omega ,\mathcal{F},\mathbf{P}\right) $ denoting the Wiener
space of a standard Wiener process $W$, for a deterministic function $h\in
L^{2}\left( \mathbf{R}_{+}\right) =:{{\mathcal{H}}}$, the Wiener integral $%
\int_{\mathbf{R}_{+}}h\left( s\right) dW\left( s\right) $ is also denoted by
$W\left( h\right) $. The inner product $\int_{\mathbf{R}_{+}}f\left(
s\right) g\left( s\right) ds$ will be denoted by $\left\langle
f,g\right\rangle _{{\mathcal{H}}}$.

\begin{itemize}
\item \textbf{The Wiener chaos expansion}. For every $q\geq 1$, ${\mathcal{H}%
}_{q}$ denotes the $q$th Wiener chaos of $W$, defined as the closed linear
subspace of $L^{2}(\Omega )$ generated by the random variables $%
\{H_{q}(W(h)),h\in {{\mathcal{H}}},\Vert h\Vert _{{\mathcal{H}}}=1\}$ where $%
H_{q}$ is the $q$th Hermite polynomial. All Wiener chaos random variable are
orthogonal in $L^{2}\left( \Omega \right) $. The so-called Wiener chaos
expansion is the fact that any $X\in L^{2}\left( \Omega \right) $ can be
written as
\begin{equation}
X=\mathbf{E}X+\sum_{q=0}^{\infty }X_{q}  \label{WienerChaos}
\end{equation}%
for some $X_{q}\in {\mathcal{H}}_{q}$ for every $q\geq 1$. This is
summarized in the direct-orthogonal-sum notation $L^{2}\left( \Omega \right)
=\oplus _{q=0}^{\infty }{\mathcal{H}}_{q}$. Here ${\mathcal{H}}_{0}$ denotes
the constants.

\item \textbf{Relation with Hermite polynomials. Multiple Wiener integrals}.
The mapping ${I_{q}(h^{\otimes q}):}=q!H_{q}(W(h))$ is a linear isometry
between the symmetric tensor product ${\mathcal{H}}^{\odot q}$ (equipped
with the modified norm $\Vert .\Vert _{{\mathcal{H}}^{\odot q}}=\frac{1}{%
\sqrt{q!}}\Vert .\Vert _{{\mathcal{H}}^{\otimes q}}$) and ${\mathcal{H}}_{q}$%
. To relate this to standard stochastic calculus, one first notes that ${%
I_{q}(h^{\otimes q})}$ can be interpreted as the multiple Wiener integral of
${h^{\otimes q}}$ w.r.t. $W$. By this we mean that the Riemann-Stieltjes
approximation of such an integral converges in $L^{2}\left( \Omega \right) $
to ${I_{q}(h^{\otimes q})}$. This is an elementary fact from analysis on
Wiener space, which can also be proved using standard stochastic calculus
for square-integrable martingales, because the multiple integral
interpretation of ${I_{q}(h^{\otimes q})}$ as a Riemann-Stieltjes integral
over $\left( \mathbf{R}_{+}\right) ^{q}$ can be further shown to coincinde
with $q!$ times the iterated Itô integral over the first simplex in $\left(
\mathbf{R}_{+}\right) ^{q}$.

More generally, for $X$ and its Wiener chaos expansion (\ref{WienerChaos})
above, each term $X_{q}$ can be interpreted as a multiple Wiener intergral $%
I_{q}\left( f_{q}\right) $ for some $f_{q}\in {\mathcal{H}}^{\odot q}$.

\item \textbf{The product formula - Isometry property}. For every $f,g\in {{%
\mathcal{H}}}^{\odot q}$ the following extended isometry property holds
\begin{equation*}
E\left( I_{q}(f)I_{q}(g)\right) =q!\langle f,g\rangle _{{\mathcal{H}}%
^{\otimes q}}.
\end{equation*}%
Similarly as for ${I_{q}(h^{\otimes q})}$, this formula is established using
basic analysis on Wiener space, but it can also be proved using standard
stochastic calculus, owing to the coincidence of $I_{q}(f)$ and $I_{q}(g)$
with iterated Itô integrals. To do so, one uses Itô's version of integration
by parts, in which iterated calculations show coincidence of the expectation
of the bounded variation term with the right-hand side above. What is
typically referred to as the Product Formula on Wiener space is the version
of the above formula before taking expectations (see \cite[Section 2.7.3]%
{NP-book}). In our work, beyond the zero-order term in that formula, which
coincides with the expectation above, we will only need to know the full
formula for $q=1$, which is $I_{1}(f)I_{1}(g)=2^{-1}I_{2}\left( f\otimes
g+g\otimes f\right) +\langle f,g\rangle _{{\mathcal{H}}}$.

\item \textbf{Hypercontractivity in Wiener chaos}. For $h\in {\mathcal{H}}%
^{\otimes q}$, the multiple Wiener integrals $I_{q}(h)$, which exhaust the
set ${\mathcal{H}}_{q}$, satisfy a hypercontractivity property (equivalence
in ${\mathcal{H}}_{q}$ of all $L^{p}$ norms for all $p\geq 2$), which
implies that for any $F\in \oplus _{l=1}^{q}{\mathcal{H}}_{l}$ (i.e. in a
fixed sum of Wiener chaoses), we have
\begin{equation}
\left( E\big[|F|^{p}\big]\right) ^{1/p}\leqslant c_{p,q}\left( E\big[|F|^{2}%
\big]\right) ^{1/2}\ \mbox{ for any }p\geq 2.  \label{hypercontractivity}
\end{equation}%
It should be noted that the constants $c_{p,q}$ above are known with some
precision when $F$ is a single chaos term: indeed, by Corollary 2.8.14 in
\cite{NP-book}, $c_{p,q}=\left( p-1\right) ^{q/2}$.

\item \textbf{Malliavin derivative}. The Malliavin derivative operator $D$
on Wiener space is not needed explicitly in the paper. However, because of
the fundamental role $D$ plays in evaluating distances between random
variables, it is helpful to introduce it, to justify the estimates (\ref%
{NPobschaos}) and (\ref{optimal berry esseen}) below. For any function $\Phi
\in C^{1}\left( \mathbf{R}\right) $ with bounded derivative, and any $h\in {%
\mathcal{H}}$, the Malliavin derivative of the random variable $X:=\Phi
\left( W\left( h\right) \right) $ is defined to be consistent with the
following chain rule:%
\begin{equation*}
DX:X\mapsto D_{r}X:=\Phi ^{\prime }\left( W\left( h\right) \right) h\left(
r\right) \in L^{2}\left( \Omega \times \mathbf{R}_{+}\right) .
\end{equation*}%
A similar chain rule holds for multivariate $\Phi $. One then extends $D$ to
the so-called Gross-Sobolev subset $\mathbf{D}^{1,2}\varsubsetneqq
L^{2}\left( \Omega \right) $ by closing $D$ inside $L^{2}\left( \Omega
\right) $ under the norm defined by
\begin{equation*}
\left\Vert X\right\Vert _{1,2}^{2}=\mathbf{E}\left[ X^{2}\right] +\mathbf{E}%
\left[ \int_{\mathbf{R}_{+}}\left\vert D_{r}X\right\vert ^{2}dr\right] .
\end{equation*}%
All Wiener chaos random variable are in the domain $\mathbf{D}^{1,2}$ of $D$%
. In fact this domain can be expressed explicitly for any $X$ as in (\ref%
{WienerChaos}): $X\in \mathbf{D}^{1,2}$ if and only if $\sum_{q}q\Vert
f_{q}\Vert _{{\mathcal{H}}^{\otimes q}}^{2}<\infty $.

\item \textbf{Generator }$L$ \textbf{of the Ornstein-Uhlenbeck semigroup}.
The linear operator $L$ is defined as being diagonal under the Wiener chaos
expansion of $L^{2}\left( \Omega \right) $: ${\mathcal{H}}_{q}$ is the
eigenspace of $L$ with eigenvalue $-q$, i.e. for any $X\in {\mathcal{H}}_{q}$%
, $LX=-qX$. We have $Ker$($L)=$ ${\mathcal{H}}_{0}$, the constants. The
operator $-L^{-1}$ is the negative pseudo-inverse of $L$, so that for any $%
X\in {\mathcal{H}}_{q}$, $-L^{-1}X=q^{-1}X$. Since the variables we will be
dealing with in this article are finite sums of elements of ${\mathcal{H}}%
_{q}$, the operator $-L^{-1}$ is easy to manipulate thereon. The use of $L$
is crucial when evaluating the total variation distance between the laws of
random variables in Wiener chaos, as we will see shortly.

\item \textbf{Distances between random variables}. Recall that, if $X,Y$ are
two real-valued random variables, then the total variation distance between
the law of $X$ and the law of $Y$ is given by
\begin{equation*}
d_{TV}\left( X,Y\right) =\sup_{A\in \mathcal{B}({\mathbb{R}})}\left\vert P%
\left[ X\in A\right] -P\left[ Y\in A\right] \right\vert
\end{equation*}%
where the supremum is over all Borel sets. If $X,Y$ are two real-valued
integrable random variables, then the Wasserstein distance between the law
of $X$ and the law of $Y$ is given by
\begin{equation*}
d_{W}\left( X,Y\right) =\sup_{f\in Lip(1)}\left\vert Ef(X)-Ef(Y)\right\vert
\end{equation*}%
where $Lip(1)$ indicates the collection of all Lipschitz functions with
Lipschitz constant $\leqslant 1$. Let $N$ denote the standard normal law.

\item \textbf{Malliavin operators and distances between laws on Wiener space}%
. There are two key estimates linking total variation distance and the
Malliavin calculus, which were both obtained by Nourdin and Peccati. The
first one is an observation relating an integration-by-parts formula on
Wiener space with a classical result of Ch. Stein. The second is a
quantitatively sharp version of the famous 4th moment theorem of Nualart and
Peccati.

\begin{itemize}
\item \textbf{The observation of Nourdin and Peccati}. Let $X\in \mathbf{D}%
^{1,2}$ with $\mathbf{E}\left[ X\right] =0$ and $Var\left[ X\right] =1$.
Then (see \cite[Proposition 2.4]{NP2015}, or \cite[Theorem 5.1.3]{NP-book}),
for $f\in C_{b}^{1}\left( \mathbf{R}\right) $,%
\begin{equation*}
E\left[ Xf\left( X\right) \right] =E\left[ f^{\prime }\left( X\right)
\left\langle DX,-DL^{-1}X\right\rangle _{\mathcal{H}}\right]
\end{equation*}%
and by combining this with properties of solutions of Stein's equations, one
gets%
\begin{equation}
d_{TV}\left( X,N\right) \leqslant 2E\left\vert 1-\left\langle
DX,-DL^{-1}X\right\rangle _{\mathcal{H}}\right\vert .  \label{NPobs}
\end{equation}%
One notes in particular that when $X\in {\mathcal{H}}_{q}$, since $%
-L^{-1}X=q^{-1}X$, we obtain conveniently%
\begin{equation}
d_{TV}\left( X,N\right) \leqslant 2E\left\vert 1-q^{-1}\left\Vert
DX\right\Vert _{\mathcal{H}}^{2}\right\vert .  \label{NPobschaos}
\end{equation}%
It is this last observation which leads to a quantitative version of the
fourth moment theorem of Nualart and Peccati, which entails using Jensen's
inequality to note that the right-hand side of (\ref{NPobs}) is bounded
above by the variance of $\left\langle DX,-DL^{-1}X\right\rangle _{\mathcal{H%
}}$, and then relating that variance in the case of Wiener chaos with the
4th cumulant of $X$. However, this strategy was superceded by the following,
which has roots in \cite{BBNP}.

\item \textbf{The optimal fourth moment theorem}. For each integer $n$, let $%
X_{n}\in {\mathcal{H}}_{q}$. Assume $Var\left[ X_{n}\right] =1$ and $\left(
X_{n}\right) _{n}$ converges in distribution to a normal law. It is known
(original proof in \cite{NP2005}, known as the \emph{fourth moment theorem})
that this convergence is equivalent to $\lim_{n}\mathbf{E}\left[ X_{n}^{4}%
\right] =3$. The following optimal estimate for $d_{TV}\left( X,N\right) $,
known as the optimal fourth moment theorem, was proved in \cite{NP2015}:
with the sequence $X$ as above, assuming convergence, there exist two
constant $c,C>0$ depending only on $X$ but not on $n$, such that
\begin{equation}
c\max \left\{ \mathbf{E}\left[ X_{n}^{4}\right] -3,\left\vert \mathbf{E}%
\left[ X_{n}^{3}\right] \right\vert \right\} \leqslant d_{TV}\left(
X_{n},N\right) \leqslant C\max \left\{ \mathbf{E}\left[ X_{n}^{4}\right]
-3,\left\vert \mathbf{E}\left[ X_{n}^{3}\right] \right\vert \right\} .
\label{optimal berry esseen}
\end{equation}
\end{itemize}
\end{itemize}

Given its importance, the centered fourth moment, also known as the fourth
cumulant, of a standardized random variable, warrants the following special
notation:%
\begin{equation*}
\kappa _{4}\left( X\right) :=\mathbf{E}\left[ X^{4}\right] -3.
\end{equation*}%
Let us also recall that if $E[F_{n}]=0$, we denote $\kappa
_{3}(F_{n})=E[F_{n}^{3}]$ its third cumulant. Throughout the paper we use
the notation $N\sim \mathcal{N}(0,1)$. We also use the notation $C$ for any
positive real constant, independently of its value which may change from
line to line when this does not lead to ambiguity.

\section{General Context}

Let $G=\{G_{t},t\geq 0\}$ be an underlying Gaussian process, and let $%
X=\{X_{n}\}_{n\geq 0}$ be a mean-zero Gaussian sequence measurable with
respect to $G$. This means that $X$ is a sequence of random variables which
can be represented as $X_{t}=I_{1}(h_{t})$ for every $t\geq 0$, i.e. $X_{t}$
is a Wiener integral with respect to $G$, where $h_{t}\in {\Hf}$ with ${\Hf}$
the Hilbert space associated to $G$. In particular, we do not assume that $X$
is a stationary process. Define
\begin{equation*}
A_{n}(X):=\frac{1}{n}\sum_{i=1}^{n}E[X_{i}^{2}]\hspace{0.7cm}\text{and}%
\hspace{0.6cm}V_{n}(X):=\frac{1}{\sqrt{n}}%
\sum_{i=1}^{n}(X_{i}^{2}-E[X_{i}^{2}]).
\end{equation*}%
As we explained in the introduction, the goal is to estimate the asymptotic
variance $f$ of the sequence $X$, if it exists. Assuming it exists is the
first of the following four assumptions, i.e. $(\mathcal{H}1)$. The other
three give us various levels of additional regularity for the law of $X$ to
measure the speed of convergence of a quadratic-variations-based estimator
for $f$:

\begin{itemize}
\item[$(\mathcal{H}1)$] $E[X_{n}^{2}]\longrightarrow f$ as $n\longrightarrow
\infty $.

\item[$(\mathcal{H}2)$] $v_{n}:=E[V_{n}(X)^{2}]\longrightarrow \sigma ^{2}>0$%
, \text{as} $n\longrightarrow \infty $.

\item[$(\mathcal{H}3)$] $|E[X_{t}X_{s}]| \leqslant C\rho\left(|t-s|\right)$
where $\rho:\mathbb{R}\rightarrow\mathbb{R}$ is a symmetric positive
function such that $\rho(0)=1$, and $C$ is a positive constant.

\item[$(\mathcal{H}4)$] $\sqrt{n}|A_{n}(X)-f|\longrightarrow 0$, \text{as} $%
n\longrightarrow \infty $.
\end{itemize}

We consider the following estimator of $f$
\begin{equation}
\hat{f}_{n}(X)=\frac{1}{n}\sum_{i=1}^{n}X_{i}^{2}.
\label{estimator quadratic}
\end{equation}%
We refer to \cite{EEV} and \cite{EV} for the construction of this estimator (%
\ref{estimator quadratic}) for some classes of processes and sequences, i.e.
how it can be interpreted as a discretization of a least squares estimator
of $f$. The next two subsections study strong consistency and speed of
convergence in the asymptotic normality of $\hat{f}_{n}\left( X\right) $. We
make some comments on the roles and motivations behind the four assumptions
above.

\begin{itemize}
\item As we said, Assumption $(\mathcal{H}1)$ states that the asymptotic
variance $f$ of the sequence $X$ exists.

\item Assumption $(\mathcal{H}2)$ helps establish strong consistency of $%
\hat{f}_{n}\left( X\right) $ via the Borel-Cantelli lemma and
hypercontractivity of Wiener chaos. It states that a proper standartization
of $\hat{f}_{n}\left( X\right) $ has an asymptotic variance; it is also used
to help establish a quantitative upper bound on the total variation distance
between a standardized version of $\hat{f}_{n}\left( X\right) $ and the
standard normal law.

\item Assumption $(\mathcal{H}3)$ formalizes the idea that while the
sequence $X$ is not stationary, it may have a covariance structure which is
bounded above by one which may be comparable to a stationary one. This
assumption itself is largely used as a matter of convenience, since formally
by Schwartz's inequality it can always be assumed to hold for the trivial
case $\rho \equiv 1$. However, by making further quantitative estimates on
the rate of decay of $\rho $ to $0$, we find convenient explicit upper bound
expressions for the total variation distance of between a standardization of
$\hat{f}_{n}(X)$ and the normal law. The corresponding results are in
Theorem \ref{thm cv in law of Fn}.

\item In particular, Assumption $(\mathcal{H}4)$ quantifies the speed of
convergence of the mean of $\hat{f}_{n}(X)$ towards $f$; combined with
Assumption $(\mathcal{H}2)$ which determines the speed of convergence of the
variance of $\hat{f}_{n}(X)$, and using estimates on Wiener space, one can
arrive at a fully quantified upper bound on the Wasserstein distance between
$\sqrt{n}\left( \hat{f}_{n}(X)-f\right) $ and the law $\mathcal{N}\left(
0,\sigma ^{2}\right) $. The corresponding results are in Theorem \ref{berry
ess. theta gen. case}.

\item The last point above is significant because Theorem \ref{berry ess.
theta gen. case} does not rely on using $v_{n}$, which is not observed, to
standardize $\hat{f}_{n}(X)$. This theorem also decomposes the distance to
the limiting normal law into the sum of three explicit terms: one to account
for the bias of $\hat{f}_{n}(X)$ (from Assumption $(\mathcal{H}4)$), one to
account for the speed of convergence to the asymptotic variance (from
Assumption $(\mathcal{H}2)$), and one from the analysis on Wiener space
which uses the speed of decay of the correlations of $X$ (from Assumption $(%
\mathcal{H}3)$).
\end{itemize}

\subsection{Strong consistency}

The following theorem provides sufficient assumptions to obtain the
estimator $\hat{f}_{n}(X)$'s strong consistency, i.e. almost sure
convergence to $f$.

\begin{theorem}
\label{Thm consistency} Assume that $(\mathcal{H}1)$ and $(\mathcal{H}2)$
hold. Then
\begin{equation}
\hat{f}_{n}(X)\longrightarrow f  \label{consistency
general case}
\end{equation}%
almost surely as $n\longrightarrow \infty $.
\end{theorem}

\begin{proof}
It is clear that
\begin{equation*}
\hat{f}_{n}(X)=\frac{V_{n}(X)}{\sqrt{n}}+A_{n}(X).
\end{equation*}%
The hypothesis $(\mathcal{H}1)$ and the convergence of Cesàro means imply
that, as $n\longrightarrow \infty $
\begin{equation*}
A_{n}(X)\longrightarrow f.
\end{equation*}%
In order to prove (\ref{consistency general case}) it remains to check that,
as $n\longrightarrow \infty $, $\frac{V_{n}(X)}{\sqrt{n}}\longrightarrow 0$
almost surely. According to $(\mathcal{H}2)$, there exists $C>0$ such that
for all $n>1$
\begin{equation*}
\left( E\left[ \left\vert \frac{V_{n}(X)}{\sqrt{n}}\right\vert ^{2}\right]
\right) ^{1/2}\leqslant \frac{C}{\sqrt{n}}.
\end{equation*}%
Now, using the hypercontractivity property (\ref{hypercontractivity}) and
Lemma \ref{Borel-Cantelli} the result is obtained.
\end{proof}

\subsection{Asymptotic normality}

In this section we study the asymptotic normality of $\hat{f}_{n}(X)$. By
the product formula, we can write
\begin{equation*}
V_{n}(X)=\frac{1}{\sqrt{n}}\sum_{k=1}^{n}(X_{k}^{2}-E[X_{k}^{2}])=\frac{1}{%
\sqrt{n}}\sum_{k=1}^{n}I_{2}(f_{k}^{\otimes 2}).
\end{equation*}%
Set
\begin{equation*}
v_{n}(X):=E[V_{n}(X)^{2}]=\frac{2}{n}\sum_{j,k=1}^{n}\left(
E[X_{j}X_{k}]\right) ^{2},
\end{equation*}%
where the second equality is from elementary covariance calculation, and
\begin{equation*}
F_{n}(X):=\frac{V_{n}(X)}{\sqrt{v_{n}(X)}}=I_{2}(g_{n})
\end{equation*}%
where
\begin{equation*}
g_{n}:=\left( nv_{n}(X)\right) ^{-1/2}\sum_{k=1}^{n}f_{k}^{\otimes 2}.
\end{equation*}

Let us estimate the third cumulant of $F_{n}(X)$. We have for every $n\geq 1$%
, by the observation of Nourdin and Peccati (or using calculus on Wiener
chaos), we have
\begin{equation*}
\kappa _{3}(F_{n}(X))=2E[F_{n}(X)\Gamma _{1}(F_{n}(X))],
\end{equation*}%
where $\Gamma _{1}(F_{n}(X))=2I_{2}(g_{n}\widetilde{\otimes }%
_{1}g_{n})+2\Vert g_{n}\Vert _{\Hf^{\otimes 2}}^{2}$. Assume that $(\mathcal{%
H}3)$ holds. We then have
\begin{align}
\kappa _{3}(F_{n}(X))& =8\langle g_{n},g_{n}\widetilde{\otimes }%
_{1}g_{n}\rangle _{\Hf^{\otimes 2}}^{2}  \notag \\
& =8\langle g_{n},g_{n}\otimes _{1}g_{n}\rangle _{\Hf^{\otimes 2}}^{2}
\notag \\
& =\frac{8}{(nv_{n}(X))^{3/2}}\sum_{i,j,k=1}^{n}\langle f_{i},f_{k}\rangle _{%
\Hf}\langle f_{i},f_{j}\rangle _{\Hf}\langle f_{k},f_{j}\rangle _{\Hf}
\notag \\
& =\frac{8}{(nv_{n}(X))^{3/2}}%
\sum_{i,j,k=1}^{n}E[X_{i}X_{k}]E[X_{i}X_{j}]E[X_{j}X_{k}]  \notag
\end{align}%
Therefore,
\begin{align}
\left\vert \kappa _{3}(F_{n}(X))\right\vert & \unlhd \frac{1}{%
(nv_{n}(X))^{3/2}}\sum_{i,j,k=1}^{n}\left\vert \rho (i-k)\rho (i-j)\rho
(j-k)\right\vert  \notag \\
& \unlhd \frac{1}{v_{n}(X)^{3/2}\sqrt{n}}\left( \sum_{|k|<n}|\rho
(k)|^{3/2}\right) ^{2}.  \label{third cum}
\end{align}%
The last equality follows from the same argument as for the proof of \cite[%
Proposition 6.3]{BBNP}. The symbol $\unlhd $ means we omit multiplicative
universal constants.

Now for the fourth cumulant, similarly, we have
\begin{align}
\kappa _{4}(F_{n}(X))& =\frac{1}{v_{n}(X)^{2}n^{2}}%
\sum_{k_{1},k_{2},k_{3},k_{4}=1}^{n}E[X_{k_{1}}X_{k_{2}}]E[X_{k_{2}}X_{k_{3}}]E[X_{k_{3}}X_{k_{4}}]E[X_{k_{4}}X_{k_{1}}]
\notag \\
& \unlhd \frac{1}{v_{n}(X)^{2}n^{2}}\sum_{k_{1},k_{2},k_{3},k_{4}=1}^{n}%
\left\vert \rho (k-l)\rho (i-j)\rho (k-i)\rho (l-j)\right\vert  \notag \\
& \unlhd \frac{1}{v_{n}(X)^{2}n}\left( \sum_{|k|<n}|\rho (k)|^{\frac{4}{3}%
}\right) ^{3}  \label{fourth cum}
\end{align}%
where again the last inequality comes from a similar argument as for the
proof of \cite[Proposition 6.4]{BBNP}.

\begin{remark}
In \cite[Proposition 6.3, Proposition 6.4]{BBNP} the sequence $%
\{X_{n}\}_{n\geq 0}$ is a centered stationary Gaussian sequence. In our
case, it is not necessarily stationary. Our Hypothesis $(\mathcal{H}3)$ is
sufficient to avoid the assumption of stationarity.

Also note that there is no need to take the absolute value of $\kappa _{4}$
because the fourth cumulant of a variable in a fixed chaos ($F_{n}$ is in
the 2nd chaos) is known to have a positive 4th cumulant.
\end{remark}

\begin{theorem}
\label{thm cv in law of Fn} Assume that the hypothesis $(\mathcal{H}3)$
holds. Then for all $n>1$,
\begin{equation}
d_{TV}(F_{n}(X),N)\leqslant C\max \left\{ \frac{1}{v_{n}(X)^{3/2}\sqrt{n}}%
\left( \sum_{|k|<n}|\rho (k)|^{3/2}\right) ^{2},\frac{1}{v_{n}(X)^{2}n}%
\left( \sum_{|k|<n}|\rho (k)|^{\frac{4}{3}}\right) ^{3}\right\} .
\label{berry esseen
general}
\end{equation}%
In particular, if for some $\beta \geq 1/2$, we have $\rho \left( t\right)
=O\left( |t|^{-\beta }\right) $ for large $|t|$, then there exists $C>0$
such that, for all $n>1$,
\begin{equation}
d_{TV}(F_{n}(X),N)\leqslant \frac{C}{v_{n}^{2}\wedge v_{n}^{3/2}}\left\{
\begin{array}{ll}
1 & \mbox{ if }\beta =\frac{1}{2} \\
n^{\frac{3}{2}-3\beta } & \mbox{ if }\beta \in \left( \frac{1}{2},\frac{2}{3}%
\right) \\
~~ &  \\
n^{-\frac{1}{2}}\log (n)^{2} & \mbox{ if }\beta =\frac{2}{3} \\
~~ &  \\
n^{-\frac{1}{2}} & \mbox{ if }\beta >\frac{2}{3}.%
\end{array}%
\right.  \label{berry esseen
particular}
\end{equation}
\end{theorem}

\begin{proof}
The estimate (\ref{berry esseen general}) is a direct consequence of the
optimal estimate in (\ref{optimal berry esseen}), and the estimates (\ref%
{third cum}) and (\ref{fourth cum}) of the absolute third and fourth
cumulants we just computed. For (\ref{berry esseen particular}), since
\begin{equation*}
\sum_{|k|<n}|\rho (k)|^{3/2}\leqslant C\left\{
\begin{array}{ll}
\sqrt{n} & \mbox{ if }\beta =\frac{1}{2} \\
n^{1-\frac{3}{2}\beta } & \mbox{ if }\beta \in \left( \frac{1}{2},\frac{2}{3}%
\right) \\
~~ &  \\
\log n & \mbox{ if }\beta =\frac{2}{3} \\
~~ &  \\
1 & \mbox{ if }\beta >\frac{2}{3}%
\end{array}%
\right.
\end{equation*}

and
\begin{equation*}
\sum_{|k|<n}|\rho (k)|^{4/3}\leqslant C\left\{
\begin{array}{ll}
n & \mbox{ if }\beta =\frac{1}{2} \\
n^{\frac{-4}{3}\beta +1} & \mbox{ if }\beta \in \left( \frac{1}{2},\frac{3}{4%
}\right) \\
~~ &  \\
\log n & \mbox{ if }\beta =\frac{3}{4} \\
~~ &  \\
1 & \mbox{ if }\beta >\frac{3}{4}%
\end{array}%
\right.
\end{equation*}%
the desired result is obtained.
\end{proof}

\begin{remark}
The phenomenon observed in the proof of the previous theorem, by which the
estimate of the third cumulant dominates that of the fourth cumulant, was
observed originally in \cite[Proposition 6.3, Proposition 6.4]{BBNP} for
stationary sequences, and shown in \cite{NV2014} to be completely general in
the stationary case (i.e. not related to the power scale). Here we see that
stationarity is not required. This begs the question of whether the
phenomenon generalizes to other sequences in the second chaos, for instance
without using hypothesis $(\mathcal{H}3)$. We do not investigate this
question here, since it falls well outside the scope of our motivating topic
of parameter estimation.
\end{remark}

\begin{theorem}
\label{berry ess. theta gen. case} Assume that the hypothesis $(\mathcal{H}%
3) $ holds. If for some $\frac{1}{2}\leqslant \beta $, we have $\rho \left(
t\right) =O\left( |t|^{-\beta }\right) $ for large $|t|$, then for all $n>1$%
,
\begin{equation*}
d_{W}\left( \sqrt{\frac{n}{v_{n}}}\left( \hat{f}_{n}(X)-f\right) ,N\right)
\leqslant \sqrt{\frac{n}{v_{n}}}|A_{n}(X)-f|+\frac{C}{v_{n}^{2}\wedge
v_{n}^{3/2}}\left\{
\begin{array}{ll}
1 & \mbox{ if }\beta =\frac{1}{2} \\
~~ &  \\
n^{\frac{3}{2}-3\beta } & \mbox{ if }\beta \in \left( \frac{1}{2},\frac{2}{3}%
\right) \\
~~ &  \\
n^{-\frac{1}{2}}\log (n)^{2} & \mbox{ if }\beta =\frac{2}{3} \\
~~ &  \\
n^{-\frac{1}{2}} & \mbox{ if }\beta >\frac{2}{3}.%
\end{array}%
\right.
\end{equation*}%
In addition if $(\mathcal{H}2)$ holds,
\begin{eqnarray*}
d_{W}\left( \frac{\sqrt{n}}{\sigma }\left( \hat{f}_{n}(X)-f\right) ,N\right)
&\leqslant &C\left( \sqrt{n}|A_{n}(X)-f|+|v_{n}-\sigma ^{2}|\right) \\
&&+C\left\{
\begin{array}{ll}
n^{\frac{3}{2}-3\beta } & \mbox{ if }\beta \in \left( \frac{1}{2},\frac{2}{3}%
\right) \\
~~ &  \\
n^{-\frac{1}{2}}\log (n)^{2} & \mbox{ if }\beta =\frac{2}{3} \\
~~ &  \\
n^{-\frac{1}{2}} & \mbox{ if }\beta >\frac{2}{3}.%
\end{array}%
\right.
\end{eqnarray*}%
In particular, if the assumptions $(\mathcal{H}2)$-$(\mathcal{H}4)$ hold and
for some $\beta >1/2$, we have $\rho \left( t\right) =O\left( |t|^{-\beta
}\right) $ for large $|t|$, then, as $n\rightarrow \infty $
\begin{equation*}
\sqrt{n}(\hat{f}_{n}(X)-f)\overset{law}{\longrightarrow }\mathcal{N}%
(0,\sigma ^{2}).
\end{equation*}
\end{theorem}

\begin{proof}
Theorem \ref{berry ess. theta gen. case} is a direct consequence of Theorem %
\ref{thm cv in law of Fn}, standard properties of the Wasserstein distance,
and the fact that
\begin{equation*}
\sqrt{\frac{n}{v_{n}}}\left( \hat{f}_{n}(X)-f\right) =F_{n}+\sqrt{\frac{n}{%
v_{n}}}|A_{n}(X)-\mu (\theta )|.
\end{equation*}
\end{proof}

\begin{remark}
The assumption that $|E[X_{t}X_{s}]|\leqslant C\left\vert t\right\vert
^{-\beta }$ for some $\beta >1/2$ corresponds to the notion of moderating
how long the memory of the data might be. For instance, when $X$ represents
the discrete-time increments of a process based on fBm with Hurst parameter $%
H\in \left( 0,1\right) $, we expect that $\beta =2-2H$. The restriction $%
\beta >1/2$ means $H<3/4$, a well-known threshold for the limit of validity
of central limit theorems for quadratic variations of long-memory processes.
See for instance the excellent treatment of the classical Breuer-Major
theorem in \cite[Chapter 7]{NP-book}. The above theorem shows that the speed
of convergence in the CLT reaches the so-called Berry-Esséen rate of $1/%
\sqrt{n}$ as soon as $\beta >2/3$, as long as the terms coming from the bias
and variance estimates, namely $\sqrt{n}|A_{n}(X)-f|$ and $|v_{n}-\sigma
^{2}|$, are no greater than that same order $1/\sqrt{n}$. The threshold $%
\beta >2/3$ coincides with $H<2/3$ when one translates into the
Hurst-parameter memory scale; this had already been identified for the
canonical stationary case of fGn in \cite{BBNP}, the paper which was the
precursor to the optimal fourth moment theorem in \cite{NP2015}.
\end{remark}

\section{Application to Gaussian Ornstein Uhlenbeck processes}

In this section we will apply the above results to Ornstein-Uhlenbeck
processes driven by a Gaussian process which does not necessarily have
stationary increments. More precisely we will study the cases which
correspond to sub-fractional Brownian motion and to bifractional Brownian
motion.

We consider the Gaussian Ornstein-Uhlenbeck process $X=\{X_{t}\}_{t\geq 0}$
defined by the following linear stochastic differential equation
\begin{equation}
dX_{t}=-\theta X_{t}dt+dG_{t},\hspace{1cm}X_{0}=0;\hspace{1cm}  \label{GOU}
\end{equation}%
where $G$ is an arbitrary mean-zero Gaussian process which has Hölder
continuous paths of strictly positive order, and $\theta >0$ is an unknown
parameter. Our goal is to estimate $\theta $ under the discrete observations
$\{X_{1},\ldots ,X_{n}\}$ as $n\rightarrow \infty $. The equation (\ref{GOU}%
) has the following explicit solution
\begin{equation*}
X_{t}=\e^{-\theta t}\int_{0}^{t}\e^{\theta s}dG_{s},\hspace{1cm}t\geq 0
\end{equation*}%
where the integral can be understood in the Young sense, since $\e^{\theta
s} $ is a Lipshitz function. The Young sense coincide in this case with
Wiener integral sense here. As mentioned in the introduction, we will
consider two different cases for $G$: the sub-fractional Brownian motion,
and the bi-fractional Brownian motion.

\subsection{Sub-fractional Brownian motion}

Consider a sub-fractional Brownian motion (sfBm) $S^{H}:=\{S_{t}^{H},t\geq
0\}$ with parameter $H\in (0,1)$ : this is the mean-zero Gaussian process
with covariance function
\begin{equation*}
R_{S^{H}}(s,t):=E\left( S_{t}^{H}S_{s}^{H}\right) =t^{2H}+s^{2H}-\frac{1}{2}%
\left( (t+s)^{2H}+|t-s|^{2H}\right) .
\end{equation*}%
Note that, when $H=\frac{1}{2}$, $S^{\frac{1}{2}}$ is a standard Brownian
motion. By Kolmogorov's continuity criterion and the fact that
\begin{equation*}
E\left( S_{t}^{H}-S_{s}^{H}\right) ^{2}\leqslant (2-2^{2H-1})|s-t|^{2H};\
s,\ t\geq ~0,
\end{equation*}%
we deduce that $S^{H}$ has Hölder continuous paths of order $H-\varepsilon $%
, for every $\varepsilon \in (0,H)$. In this section we replace $G$ given in
(\ref{GOU}) by $S^{H}$. More precisely, we will estimate the drift parameter
$\theta $ in the following equation
\begin{equation}
dX_{t}=-\theta X_{t}dt+dS_{t}^{H},\quad X_{0}=0.  \label{subFOU}
\end{equation}

\begin{proposition}
\label{rate cv of 2 mmt sbfOU} Suppose that $H\in (0,1)$. Set
\begin{equation}
f_{H}(\theta ):=\frac{H\Gamma (2H)}{\theta ^{2H}}.  \label{fH}
\end{equation}%
Then for every $t>0$,
\begin{equation*}
|E[X_{t}^{2}]-f_{H}(\theta )|\leqslant Ct^{2H-2}.
\end{equation*}%
Hence
\begin{equation*}
\sqrt{n}|A_{n}(X)-f_{H}(\theta )|\leqslant \frac{C}{\sqrt{n}}%
\sum_{i=1}^{n}i^{2H-2}\leqslant C\left\{
\begin{array}{ll}
n^{-\frac{1}{2}} & \mbox{ if \hspace{0.2cm}}0<H<1/2 \\
~~ &  \\
n^{-\frac{1}{2}}log(n) & \mbox{ if \hspace{0.2cm}}H=1/2 \\
~~ &  \\
n^{2H-3/2} & \mbox{ if \hspace{0.2cm}}H>1/2.%
\end{array}%
\right.
\end{equation*}%
In particular, if $0<H<3/4$, the hypothesis $(\mathcal{H}4)$ holds.
\end{proposition}

\begin{proof}
Since $X_{t}=\e^{-\theta t}\int_{0}^{t}\e^{\theta s}dS_{s}^{H},\quad t\geq 0$%
, we can write (see \cite{EEO})
\begin{equation*}
E[X_{t}^{2}]=\Delta _{H}(t)+\theta I_{2H}(t)-\frac{\theta ^{2}}{2}J_{2H}(t),
\end{equation*}%
where
\begin{equation*}
\Delta _{H}(t)=2He^{-2\theta t}\int_{0}^{t}s^{2H-1}e^{\theta
s}ds-2H(2H-1)e^{-2\theta t}\int_{0}^{t}dse^{\theta s}\int_{0}^{s}dre^{\theta
r}(s+r)^{2H-2},
\end{equation*}%
and
\begin{equation*}
J_{2H}(t):=\e^{-2\theta t}\int_{0}^{t}\int_{0}^{t}\e^{\theta s}\e^{\theta
r}|s-r|^{2H}drds\text{;}\qquad I_{2H}(t):=\e^{-\theta t}\int_{0}^{t}\e%
^{\theta s}(t-s)^{2H}ds.
\end{equation*}%
Hence
\begin{equation*}
|E[X_{t}^{2}]-f_{H}(\theta )|\leqslant |\Delta _{H}(t)|+\theta \left\vert
I_{2H}(t)-\frac{2H\Gamma (2H)}{\theta ^{2H+1}}\right\vert +\frac{\theta ^{2}%
}{2}\left\vert J_{2H}(t)-\frac{2H\Gamma (2H)}{\theta ^{2H+2}}\right\vert .
\end{equation*}%
We will check that each term of the right side of the last inequality is
less than $Ct^{2H-2}$. We can write
\begin{equation*}
\Delta _{H}(t)=2Ha_{H}(t)-2H(2H-1)b_{H}(t)
\end{equation*}%
where
\begin{equation*}
a_{H}(t):=e^{-2\theta t}\int_{0}^{t}s^{2H-1}e^{\theta s}ds;\qquad
b_{H}(t):=e^{-2\theta t}\int_{0}^{t}dse^{\theta s}\int_{0}^{s}dre^{\theta
r}(s+r)^{2H-2}.
\end{equation*}%
\newline
It is clear that
\begin{equation*}
|a_{H}(t)|\leqslant \frac{t^{2H}}{2H}e^{-\theta t}\leqslant C\e^{-\frac{%
\theta t}{2}}.
\end{equation*}%
On the other hand,
\begin{eqnarray*}
|b_{H}(t)| &=&e^{-2\theta t}\int_{0}^{t}dse^{\theta
s}\int_{0}^{s}dre^{\theta r}(s+r)^{2H-2}=\e^{-2\theta t}\int_{0}^{t}\e%
^{\theta s}ds\int_{s}^{2s}\e^{\theta (u-s)}u^{2H-2}du \\
&=&\e^{-2\theta t}\int_{0}^{2t}\e^{\theta u}u^{2H-2}\left( t\wedge u-\frac{u%
}{2}\right) du=\frac{\e^{-2\theta t}}{2}\left( \int_{0}^{t}\e^{\theta
u}u^{2H-1}du+\int_{t}^{2t}(2t-u)\e^{\theta u}u^{2H-2}du\right) \\
&\leqslant &\frac{t^{2H}}{4H}\e^{-\theta t}+t^{2H-2}\int_{t}^{2t}(2t-u)\e%
^{\theta u}du\leqslant Ct^{2H-2}.
\end{eqnarray*}%
Hence, for every $t>0$, $|\Delta _{H}(t)|\leqslant Ct^{2H-2}$. We also have
for $t\geq 0$
\begin{eqnarray*}
\left\vert J_{2H}(t)-\frac{\Gamma (2H+1)}{\theta ^{2H+2}}\right\vert
&=&\left\vert \frac{1}{\theta }\left( \int_{0}^{t}u^{2H}\e^{-\theta u}du-\e%
^{-2\theta t}\int_{0}^{t}u^{2H}\e^{\theta u}du\right) -\frac{\Gamma (2H+1)}{%
\theta ^{2H+2}}\right\vert \\
&\leqslant &\frac{1}{\theta }\left\vert \int_{t}^{\infty }u^{2H}\e^{-\theta
u}du-\e^{-2\theta t}\int_{0}^{t}u^{2H}\e^{\theta u}du\right\vert \\
&\leqslant &\frac{\e^{-\theta t/2}}{\theta }\left( \int_{t}^{\infty }u^{2H}\e%
^{-\theta u/2}du+\int_{0}^{t}u^{2H}\e^{-\theta u/2}du\right) \\
&\leqslant &C\e^{-\theta t/2}.
\end{eqnarray*}%
For the last term, let $t\geq 0$
\begin{eqnarray*}
\left\vert I_{2H}(t)-\frac{\Gamma (2H+1)}{\theta ^{2H+1}}\right\vert
&=&\left\vert \int_{0}^{t}u^{2H}\e^{-\theta u}du-\frac{\Gamma (2H+1)}{\theta
^{2H+1}}\right\vert \\
&=&\int_{t}^{+\infty }u^{2H}\e^{-\theta u}du\leqslant \e^{-\theta
t/2}\int_{t}^{+\infty }\e^{-\theta u/2}u^{2H}du\leqslant C\e^{-\theta t/2}
\end{eqnarray*}%
which completes the proof.
\end{proof}

\begin{proposition}
\label{H3 subfBm} For every $H\in(0,1)$, the hypothesis $(\mathcal{H}3)$
holds. More precisely, we have for large $|t-s|$
\begin{eqnarray*}
\left|E[X_{t}X_{s}]\right|\leqslant C |t-s|^{2H-2}.
\end{eqnarray*}
\end{proposition}

\begin{proof}
If $H=\frac{1}{2}$, it is easy to see that $\left\vert
E[X_{t}X_{s}]\right\vert \leqslant \frac{1}{2\theta }e^{-\theta
(t-s)}\leqslant C|t-s|^{2H-2}$ for $|t-s|>0$.\newline
Now, suppose that $H\in (0,\frac{1}{2})\cup (\frac{1}{2},1)$. Thanks to
Lemma \ref{calculcov} we get
\begin{eqnarray*}
E[X_{t}X_{s}] &=&\e^{-\theta (t-s)}E[X_{s}^{2}]+H(2H-1)\e^{-\theta t}\e%
^{-\theta s}\int_{s}^{t}dv\e^{\theta v}\int_{0}^{s}du\e^{\theta
u}(v-u)^{2H-2} \\
&&\quad -H(2H-1)\e^{-\theta t}\e^{-\theta s}\int_{s}^{t}dv\e^{\theta
v}\int_{0}^{s}du\e^{\theta u}(u+v)^{2H-2}.
\end{eqnarray*}%
Hence,
\begin{equation*}
|E[X_{s}X_{t}]|\leqslant \e^{-\theta (t-s)}E[X_{s}^{2}]+2\e^{-\theta t}\e%
^{-\theta s}H|2H-1|\int_{s}^{t}\int_{0}^{s}\e^{\theta u}\e^{\theta
v}(v-u)^{2H-2}dudv.
\end{equation*}%
Define for every $\alpha\in(0,1)$
\begin{equation}
Z^{\alpha}_{t}:=\e^{-\theta t}\int_{-\infty }^{t}\e^{\theta
u}dB_{u}^{\alpha},\quad t\in \mathbb{R}  \label{stationary-FOU}
\end{equation}%
which is a stationary Gaussian process, where $B^{\alpha}$ is a fractional
Brownian motion with Hurst parameter $\alpha\in (0,1)$. Then, we can write
for $s<t$,
\begin{eqnarray*}
&&H(2H-1)\e^{-\theta t}\e^{-\theta s}\int_{s}^{t}\int_{0}^{s}\e^{\theta u}\e%
^{\theta v}(v-u)^{2H-2}dudv \\
&=&H(2H-1)\e^{-\theta (t-s)}\int_{0}^{t-s}dy\e^{\theta y}\int_{-s}^{0}dx\e%
^{\theta x}(y-x)^{2H-2} \\
&=&\e^{-\theta (t-s)}E\left[ \left( \int_{0}^{t-s}\e^{\theta
y}dB_{y}^{H}\right) \left(\int_{-s}^{0}\e^{\theta x}dB_{x}^{H}\right) \right]
=E\left[ \left( Z^H_{t-s}-\e^{-\theta (t-s)}Z^H_{0}\right) \left( Z^H_{0}-\e%
^{-\theta s}Z^H_{-s}\right) \right] \\
&=&E[Z^H_{t-s}Z^H_{0}]-\e^{-\theta s}E[Z^H_{0}Z^H_{t}]-\e^{-\theta
(t-s)}E[\left(Z^H_{0}\right)^{2}]+\e^{-\theta t}E[Z^H_{0}Z^H_{s}] \\
&\leqslant &C\left( (t-s)^{2H-2}+\e^{-\theta (t-s)}\right) .
\end{eqnarray*}%
The last inequality comes from the fact that for large $r>0$ $%
E[Z^H_{r}Z^H_{0}]\leqslant C|r|^{2H-2}$ (see \cite{CKM}, or \cite{EV}).
\end{proof}

Define the following rates of convergence,
\begin{eqnarray}
\varphi_{\alpha}(n) =\left\{
\begin{array}{ll}
n^{-1} & \mbox{ if } 0 < \alpha < 1/2 \\
~~ &  \\
~~ &  \\
n^{2\alpha - 2} & \mbox{ if } 1/2 \leqslant \alpha < 1,%
\end{array}
\right.
\end{eqnarray}
and
\begin{eqnarray}
\psi_{\alpha}(n) =\left\{
\begin{array}{ll}
n^{-1} & \mbox{ if } 0 < \alpha < 1/2 \\
~~ &  \\
~~ &  \\
n^{4\alpha - 3} & \mbox{ if } 1/2 \leqslant \alpha < 3/4.%
\end{array}
\right.
\end{eqnarray}

\begin{proposition}
\label{rate cv qua-varia sbfOU} Let $0<H<3/4$. Define
\begin{equation}
\sigma _{H}^{2}:=\rho _{H}(0)^{2}+2\sum_{i\in \mathbb{Z}\backslash
\{0\}}\rho _{H}(i)^{2}  \label{sigmaH}
\end{equation}%
where $\rho _{H}(k):=E[Z^H_{k}Z^H_{0}]$, $k\in \mathbb{N}$ and $Z$ is the
process given in (\ref{stationary-FOU}). Note that $\rho
_{H}(0)=f_{H}(\theta )$. Then
\begin{equation*}
|E[V_{n}(X)^{2}]-\sigma _{H}^{2}|\leqslant C\psi _{H}(n).
\end{equation*}%
In particular, the hypothesis $(\mathcal{H}2)$ holds. \newline
If $H=3/4$, we have
\begin{equation*}
\left\vert \frac{E[V_{n}(X)^{2}]}{\log (n)}-\frac{9}{16\theta ^{4}}%
\right\vert \leqslant C\log (n)^{-1}.
\end{equation*}
\end{proposition}

\begin{proof}
See Appendix.
\end{proof}

Propositions \ref{rate cv of 2 mmt sbfOU} and \ref{rate cv qua-varia sbfOU}
lead to the assumptions $(\mathcal{H}1)$ and $(\mathcal{H}2)$. Then,
applying Theorem \ref{Thm consistency} we obtain the strong consistency of
the estimator $\hat{f}_{n}(X)$ of the form (\ref{estimator quadratic}).

\begin{theorem}
Let $0<H<1$. Then we have
\begin{equation*}
\hat{f}_{n}(X)\longrightarrow f_{H}(\theta )=\frac{H\Gamma (2H)}{\theta ^{2H}%
}
\end{equation*}%
almost surely as $n\longrightarrow \infty $.
\end{theorem}

Now, we will study the asymptotic normality of the estimator $\hat{f}_{n}(X)$
when $0<H\leqslant \frac{3}{4}$. Using (\ref{berry esseen particular}) with $%
\beta =2-2H$ we obtain the following result.

\begin{proposition}
\label{ratesFn sbfbm} If $0<H\leqslant \frac{3}{4}$, then
\begin{equation*}
d_{TV}(F_{n}(X),N)\leqslant C\left\{
\begin{array}{ll}
n^{-\frac{1}{2}} & \mbox{ if }H\in \left( 0,\frac{2}{3}\right) \\
~~ &  \\
n^{-\frac{1}{2}}\log (n)^{2} & \mbox{ if }H=\frac{2}{3} \\
~~ &  \\
n^{6H-\frac{9}{2}} & \mbox{ if }H\in \left( \frac{2}{3},\frac{3}{4}\right)
\\
~~ &  \\
\log (n)^{-3/2} & \mbox{ if }H=\frac{3}{4}.%
\end{array}%
\right.
\end{equation*}
\end{proposition}

Combining this with Propositions \ref{rate cv of 2 mmt sbfOU} and \ref{rate
cv qua-varia sbfOU} we deduce the result.

\begin{theorem}
If $0<H<3/4$, then
\begin{equation*}
d_{W}\left( \frac{\sqrt{n}}{\sigma _{H}}(\hat{f}_{n}(X)-f_{H}(\theta
)),N\right) \leqslant C\sqrt{\psi _{H}(n)}
\end{equation*}%
and if $H=3/4$, we have
\begin{equation*}
d_{W}\left( \frac{\sqrt{n}(\hat{f}_{n}(X)-f_{H}(\theta ))}{\sigma _{H}\sqrt{%
\log (n)}},N\right) \leqslant C\log (n)^{-1/2}.
\end{equation*}%
In particular, if $0<H<3/4$, then, as $n\rightarrow \infty $
\begin{equation*}
\sqrt{n}(\hat{f}_{n}(X)-f_{H}(\theta ))\overset{law}{\longrightarrow }%
\mathcal{N}(0,\sigma _{H}^{2})
\end{equation*}%
and if $H=3/4$, then, as $n\rightarrow \infty $
\begin{equation*}
\frac{\sqrt{n}(\hat{f}_{n}(X)-f_{H}(\theta ))}{\sqrt{\log (n)}}\overset{law}{%
\longrightarrow }\mathcal{N}(0,\sigma _{H}^{2}).
\end{equation*}
\end{theorem}

\subsection{Bifractional Brownian motion}

In this section we suppose that $G$ given in (\ref{GOU}) is a bifractional
Brownian motion (bifBm) $B^{H,K}$ with parameters $H\in (0,1)$ and $K\in
(0,1]$. The $B^{H,K}:=\{B_{t}^{H,K},t\geq 0\}$ is the mean-zero Gaussian
process with covariance function
\begin{equation*}
E(B_{s}^{H,K}B_{t}^{H,K})=\frac{1}{2^{K}}\left( \left( t^{2H}+s^{2H}\right)
^{K}-|t-s|^{2HK}\right) .
\end{equation*}%
The case $K=1$ corresponds to the fBm with Hurst parameter $H$. The process $%
B^{H,K}$ verifies,
\begin{equation*}
E\left( \left\vert B_{t}^{H,K}-B_{s}^{H,K}\right\vert ^{2}\right) \leqslant
2^{1-K}|t-s|^{2HK},
\end{equation*}%
so $B^{H,K}$ has $(HK-\varepsilon )-$Hölder continuous paths for any $%
\varepsilon \in (0,HK)$ thanks to Kolmogorov's continuity criterion.

\begin{proposition}
\label{rate cv of 2 mmt bifOU} Assume that $H\in (0,1)$ and $K\in (0,1]$.
Then we have for large $t>0$
\begin{equation*}
|E[X_{t}^{2}]-f_{H,K}(\theta )|\leqslant Ct^{2HK-2}
\end{equation*}%
where
\begin{equation}
f_{H,K}(\theta ):=2^{1-K}HK\Gamma (2HK)/\theta ^{2HK}.  \label{fHK}
\end{equation}%
Then
\begin{equation*}
\sqrt{n}|A_{n}(X)-\mu (\theta )|\leqslant C\left\{
\begin{array}{ll}
n^{-\frac{1}{2}} & \mbox{ if }0<HK<1/2 \\
~~ &  \\
n^{2HK-3/2} & \mbox{ if }HK\geq 1/2.%
\end{array}%
\right.
\end{equation*}%
In particular, if $HK<3/4$, the hypothesis $(\mathcal{H}4)$ holds.
\end{proposition}

\begin{proof}
From \cite{EEO} and the fact that $X_{t}=\e^{-\theta t}\int_{0}^{t}\e%
^{\theta s}dB_{s}^{H,K},\quad t\geq 0$, we can write
\begin{equation*}
E\left[ X_{t}^{2}\right] =\Delta _{H,K}(t)+2^{1-K}\theta
I_{2HK}(t)-2^{-K}\theta ^{2}J_{2HK}(t),
\end{equation*}%
where
\begin{eqnarray*}
\Delta _{H,K}(t) &=&2^{2-K}HK\e^{-2\theta t}\int_{0}^{t}\e^{\theta
s}s^{2HK-1}ds \\
&&+2^{3-K}H^{2}K(K-1)\e^{-2\theta
t}\int_{0}^{t}\int_{0}^{s}(sr)^{2H-1}(s^{2H}+r^{2H})^{K-2}\e^{\theta r}\e%
^{\theta s}drds
\end{eqnarray*}%
and
\begin{equation*}
J_{2HK}(t)=\e^{-2\theta t}\int_{0}^{t}\int_{0}^{s}\e^{\theta s}\e^{\theta
r}(s-r)^{2HK}drds;\qquad I_{2HK}(t)=\e^{-\theta t}\int_{0}^{t}\e^{\theta
s}(t-s)^{2HK}ds.
\end{equation*}%
Hence
\begin{equation*}
|E[X_{t}^{2}]-\mu (\theta )|\leqslant |\Delta _{H}(t)|+\frac{\theta }{2^{K-1}%
}\left\vert I_{2HK}(t)-\frac{2HK\Gamma (2HK)}{\theta ^{2HK+1}}\right\vert +%
\frac{\theta ^{2}}{2^{K}}\left\vert J_{2HK}(t)-\frac{2HK\Gamma (2HK)}{\theta
^{2HK+2}}\right\vert .
\end{equation*}%
We will check that each term of the right-hand side is less than $Ct^{2HK-2}$%
. We can write%
\begin{equation*}
\Delta _{H,K}(t)=2^{2-K}HKa_{H,K}(t)+2^{3-K}H^{2}K(K-1)b_{H,K}(t).
\end{equation*}%
where
\begin{equation*}
a_{H,K}(t):=\e^{-2\theta t}\int_{0}^{t}\e^{\theta s}s^{2HK-1}ds;\qquad
b_{H,K}(t):=\e^{-2\theta
t}\int_{0}^{t}\int_{0}^{s}(sr)^{2H-1}(s^{2H}+r^{2H})^{K-2}\e^{\theta r}\e%
^{\theta s}drds.
\end{equation*}%
It is easy to prove that{}%
\begin{equation*}
a_{H,K}(t)\leqslant C\e^{-\theta t/2}.
\end{equation*}%
On the other hand, using $x^{2}+y^{2}\geq 2|xy|,\ x,y\in \mathbb{R}$, we get
\begin{eqnarray*}
b_{H,K}(t) &=&\e^{-2\theta
t}\int_{0}^{t}\int_{0}^{s}(sr)^{2H-1}(s^{2H}+r^{2H})^{K-2}\e^{\theta r}\e%
^{\theta s}drds \\
&\leqslant &\e^{-2\theta
t}2^{K-2}\int_{0}^{t}\int_{0}^{s}(sr)^{2H-1}s^{HK-1}r^{HK-1}\e^{\theta r}\e%
^{\theta s}drds \\
:= &&b_{1}(t)+b_{2}(t)
\end{eqnarray*}%
where
\begin{equation*}
b_{1}(t)=\e^{-2\theta t}\int_{0}^{\frac{t}{2}}\int_{0}^{s}(sr)^{HK-1}\e%
^{\theta r}\e^{\theta s}drds;\quad b_{2}(t)=\e^{-2\theta t}\int_{\frac{t}{2}%
}^{t}\int_{0}^{s}(sr)^{HK-1}\e^{\theta r}\e^{\theta s}drds.
\end{equation*}%
It is easy to see that
\begin{equation*}
|b_{1}(t)|\leqslant Ct^{2HK}\e^{-\theta t}\leqslant C\e^{-\frac{\theta t}{2}%
},
\end{equation*}%
and
\begin{eqnarray*}
|b_{2}(t)| &\leqslant &\left( \frac{t}{2}\right) ^{HK-1}\frac{\e^{-\theta t}%
}{\theta }\int_{0}^{t}r^{HK-1}\e^{\theta r}dr \\
&\leqslant &\left( \frac{t}{2}\right) ^{HK-1}\frac{\e^{-\theta t}}{\theta }%
\left( \int_{0}^{\frac{t}{2}}r^{HK-1}\e^{\theta r}dr+\int_{\frac{t}{2}%
}^{t}r^{HK-1}\e^{\theta r}dr\right) \\
&\leqslant &\left( \frac{t}{2}\right) ^{2HK-1}\frac{e^{-\theta t/2}}{\theta }%
+\frac{1}{\theta ^{2}}\left( \frac{t}{2}\right) ^{2HK-2} \\
&\leqslant &Ct^{2HK-2}.
\end{eqnarray*}%
We deduce that%
\begin{equation*}
\Delta _{H,K}(t)\leqslant Ct^{2HK-2}.
\end{equation*}%
Moreover, by a similar argument as in the proof of Proposition \ref{rate cv
of 2 mmt sbfOU} we have
\begin{equation*}
|I_{2HK}(t)-\frac{2HK\Gamma (2HK)}{\theta ^{2HK+1}}|\leqslant e^{-\theta
t/2}\left(\frac{2}{\theta }\right)^{2HK+1}\Gamma (2HK+1),
\end{equation*}%
and
\begin{equation*}
|J_{2HK}(t)-\frac{2HK\Gamma (2HK)}{\theta ^{2HK+2}}|\leqslant \left(\frac{2}{%
\theta }\right)^{2HK+2}\Gamma (2HK+1)\e^{-\theta t/2}
\end{equation*}%
which completes the proof.
\end{proof}

\begin{proposition}
\label{H3 bifBm} For all fixed $(H,K)\in(0,1)\times(0,1]$, with $%
HK\neq\frac12$, the hypothesis $(\mathcal{H}3)$ holds. More precisely, we
have for large $|t-s|$,
\begin{eqnarray*}
\left|E[X_{t}X_{s}]\right|\leqslant C \left\{
\begin{array}{ll}
|t-s|^{2HK-2H-1} & \mbox{ if } 0<HK < 1/2 \\
~~ &  \\
|t-s|^{2HK-2} & \mbox{ if } 1/2 < HK <1.%
\end{array}
\right.
\end{eqnarray*}
\end{proposition}

\begin{proof}
Let $s<t$. Using Lemma \ref{calculcov} we get
\begin{eqnarray*}
E[X_{t}X_{s}] &=&\e^{-\theta (t-s)}E[X_{s}^{2}]-2^{2-K}K(1-K)\e^{-\theta t}\e%
^{-\theta s}\int_{s}^{t}\e^{\theta v}\int_{0}^{s}\e^{\theta
u}(uv)^{2H-1}(u^{2H}+v^{2H})^{K-2}dudv \\
&&\quad +2^{1-K}HK(2HK-1)\e^{-\theta t}\e^{-\theta s}\int_{s}^{t}\e^{\theta
v}\int_{0}^{s}\e^{\theta u}(u-v)^{2HK-2}dudv.
\end{eqnarray*}%
As in the proof of Proposition \ref{H3 subfBm} we have
\begin{equation*}
\delta _{HK}:=e^{-\theta t}\e^{-\theta s}\int_{s}^{t}\e^{\theta
v}\int_{0}^{s}\e^{\theta u}(u-v)^{2HK-2}dudv\leqslant C|t-s|^{2HK-2}.
\end{equation*}%
Set
\begin{equation*}
\lambda _{H,K}:=\e^{-\theta t}\e^{-\theta s}\int_{s}^{t}\e^{\theta
v}\int_{0}^{s}\e^{\theta u}(uv)^{2H-1}(u^{2H}+v^{2H})^{K-2}dudv.
\end{equation*}%
If $H\geq \frac{1}{2}$, we have for $0\leqslant u\leqslant v$,%
\begin{equation*}
(uv)^{2H-1}(u^{2H}+v^{2H})^{K-2}\leqslant v^{2HK-2}\leqslant (v-u)^{2HK-2}.
\end{equation*}%
Thus, if we assume that $HK\neq \frac{1}{2}$, $\lambda _{H,K}\leqslant
C\delta _{HK}\leqslant C|t-s|^{2HK-2}$.

If $H<\frac{1}{2}$, it is clear that $s\rightarrow \e^{-\theta s}\int_{0}^{s}%
\e^{\theta u}u^{2H-1}du$ is bounded. This implies
\begin{equation*}
\lambda _{H,K}\leqslant C\e^{-\theta t}\int_{s}^{t}\e^{\theta
v}v^{2H-1}v^{2HK-4H}dv\leqslant Ct^{2HK-2H-1}\leqslant C(t-s)^{2HK-2H-1}
\end{equation*}%
which finishes the proof.
\end{proof}

Using the same arguments as in Proposition \ref{rate cv qua-varia sbfOU} we
obtain the following result.

\begin{proposition}
\label{rate cv qua-varia bifOU} For all fixed $0<HK<3/4$,
\begin{equation*}
|E[V_{n}(X)^{2}]-\sigma _{H,K}^{2}|\leqslant C\psi _{HK}(n)
\end{equation*}%
where
\begin{equation}
\sigma _{H,K}^{2}:=4\sum_{i\in \mathbb{N}^{\ast }}(\rho _{H,K}(i)-(1-2^{1-K})%
\e^{-\theta i}\rho _{H,K}(0))^{2}+2^{3-2K}\rho _{H,K}(0)^{2}  \label{sigmaHK}
\end{equation}%
with $\rho _{H,K}(k):=E[Z_{k}^{HK}Z_{0}^{HK}]$, $k\in \mathbb{N}$ where $%
Z^{HK}$ is the process given in (\ref{stationary-FOU}), and $\rho _{H,K}(0)=%
\frac{HK\Gamma (2HK)}{\theta ^{2HK}}.$ In particular, the hypothesis $(%
\mathcal{H}2)$ holds.

If $HK=3/4$, we have
\begin{equation*}
\left\vert \frac{E[V_{n}(X)^{2}]}{\log (n)}-\frac{9}{16\theta ^{4}}%
\right\vert \leqslant C\log (n)^{-1}.
\end{equation*}
\end{proposition}

Similarly as in Section 4.1 we obtain the following asymptotic behavior
results.

\begin{theorem}
Let $H,K\in (0,1)$. Then we have
\begin{equation}
\hat{f}_{n}(X)\longrightarrow f_{H,K}(\theta )=2^{1-K}HK\Gamma (2HK)/\theta
^{2HK}
\end{equation}%
almost surely as $n\longrightarrow \infty $.
\end{theorem}

\begin{theorem}
Let $HK\in (0,3/4)\setminus \{1/2\}$ and $N\sim \mathcal{N}(0,1)$, then
\begin{equation*}
d_{W}\left( \sqrt{n}(\hat{f}_{n}(X)-f_{H,K}(\theta ))/\sigma _{HK},N\right)
\leqslant C\sqrt{\psi _{HK}(n)}
\end{equation*}%
and if $HK=3/4$, we have
\begin{equation*}
d_{W}\left( \frac{\sqrt{n}(\hat{\theta}_{n}(X)-f_{H,K}(\theta ))}{\sigma
_{HK}\sqrt{\log (n)}},N\right) \leqslant C\log (n)^{-1/2}.
\end{equation*}%
In particular, if $HK\in (0,3/4)\setminus \{1/2\}$, we have as $n\rightarrow
\infty $
\begin{equation*}
\sqrt{n}(\hat{\theta}_{n}(X)-f_{H,K}(\theta ))\overset{law}{\longrightarrow }%
\mathcal{N}(0,\sigma _{H,K}^{2})
\end{equation*}%
and if $HK=3/4$, we have as $n\rightarrow \infty $
\begin{equation*}
\frac{\sqrt{n}(\hat{\theta}_{n}(X)-f_{H,K}(\theta ))}{\sqrt{\log (n)}}%
\overset{law}{\longrightarrow }\mathcal{N}(0,\sigma _{H,K}^{2}).
\end{equation*}
\end{theorem}

\section{Appendix}

\begin{lemma}
\label{calculcov} Let $G$ be a Gaussian process which has Hölder continuous
paths of stictly positive ordre and $X$ is the solution of the equation (\ref%
{GOU}). Define $R_{G}(s,t) = E[G_{s}G_{t}]$, and assume that $\frac{\partial
R_{G} }{\partial s }(s,r)$ and $\frac{\partial^2 R_{G} }{\partial s\partial r%
}(s,r)$ exist on $(0,\infty)\times(0,\infty)$. Then for every $0 < s < t$,
we have
\begin{eqnarray}
E[X_{s}X_{t}] = \e^{- \theta(t-s)} E[X_{s}^{2}] + \e^{- \theta t} \e^{-
\theta s} \int_{s}^{t} \e^{\theta v} \int_{0}^{s} \e^{\theta u} \frac{%
\partial^2R_{G}}{\partial u\partial v} (u,v) du dv.  \label{cov of X GOU}
\end{eqnarray}
\end{lemma}

\begin{proof}
Fix $s<t$. We have
\begin{eqnarray*}
E[X_{s}X_{t}] &=&E\left[ \left( \e^{-\theta s}\int_{0}^{s}\e^{\theta
u}dG_{u}\right) \left( \e^{-\theta t}\int_{0}^{t}\e^{\theta v}dG_{v}\right) %
\right] \\
&=&\e^{-\theta (t-s)}E[X_{s}^{2}]+\e^{-\theta t}\e^{-\theta s}E\left[
\int_{0}^{s}\e^{\theta u}dG_{u}\int_{s}^{t}\e^{\theta v}dG_{v}\right].
\end{eqnarray*}%
Moreover, by (\ref{Young}),
\begin{eqnarray*}
&&\e^{-\theta t}\e^{-\theta s}E\left[ \int_{0}^{s}\e^{\theta
u}dG_{u}\int_{s}^{t}\e^{\theta v}dG_{v}\right] \\
&=&\e^{-\theta t}\e^{-\theta s}E\left[ \left( \e^{\theta s}G_{s}-\theta
\int_{0}^{s}\e^{\theta u}G_{u}du\right) \left( \e^{\theta t}G_{t}-\e^{\theta
s}G_{s}^{H}-\theta \int_{s}^{t}\e^{\theta v}G_{v}dv\right) \right] \\
&=&\e^{-\theta t}\e^{-\theta s}\left[ \e^{\theta (t+s)}R_{G}(s,t)-\e%
^{2\theta s}R_{G}(s,s)-\theta \e^{\theta s}\int_{s}^{t}\e^{\theta
v}R_{G}(s,v)dv-\theta \e^{\theta t}\int_{0}^{s}\e^{\theta
u}R_{G}(u,t)du\right. \\
&&\left. +\theta \e^{\theta s}\int_{0}^{s}\e^{\theta u}R_{G}(u,s)du+\theta
^{2}\int_{s}^{t}\int_{0}^{s}\e^{\theta u}\e^{\theta v}R_{G}(u,v)dudv\right] .
\end{eqnarray*}%
Applying again (\ref{Young}) several times we get the last term of (\ref{cov
of X GOU}). Thus the desired result is obtained.
\end{proof}

\begin{lemma}
Let $X$ be the solution of (\ref{subFOU}), and let $\sigma_H^{2}$ be the
constant defined in Proposition \ref{rate cv qua-varia sbfOU}. If $0 < H <
3/4 $, we have

\begin{eqnarray*}
| E[V_{n}(X)^{2}] - \sigma_H^{2} | \leqslant C \psi_{H}(n)
\end{eqnarray*}
If $H = 3/4$, we have
\begin{equation*}
\left|\frac{E[V_{n}(X)^{2}]}{\log(n)}- \frac{9}{16 \theta^{4}}\right|
\leqslant C \log(n)^{-1}.
\end{equation*}
\end{lemma}

\begin{proof}
We have $V_{n}(X) = I_{2}(f_{n,2})$, with $f_{n,2} = \frac{1}{\sqrt{n}}
\sum_{i = 1}^{n} f_{i}^{\otimes 2} $. Then
\begin{eqnarray*}
E[V_{n}(X)^{2}] &=& \frac{2}{n} \sum_{k,l=1}^{n} (\langle f_{k},
f_{l}\rangle_{\Hf})^{2} = \frac{2}{n} \sum_{k,l=1}^{n} (E[X_{k}X_{l}])^{2} \\
&=& \frac{2}{n} \sum_{k =1}^{n} (E[X_{k}^{2}])^{2} +\frac{2}{n} \sum
\limits_{\underset{ k \neq l }{k,l=1}}^n (E[X_{k}X_{j}])^{2}.
\end{eqnarray*}
\end{proof}

We will need the following lemmas.

\begin{lemma}
Let $X$ be the solution of (\ref{subFOU}). Then,
\begin{eqnarray*}
\left|\frac{1}{n} \sum_{k =1}^{n} (E[X_{k}^{2}])^{2} -
f_H(\theta)^{2}\right| \leqslant C \varphi_{H}(n).
\end{eqnarray*}
\end{lemma}

\begin{proof}
By Proposition \ref{rate cv of 2 mmt sbfOU} we deduce that
\begin{eqnarray*}
\left|\frac{1}{n} \sum_{k =1}^{n} (E[X_{k}^{2}])^{2} -
f_H(\theta)^{2}\right| & \leqslant& \frac{1}{n} \sum_{k =1}^{n}
|E[X_{k}^{2}] - f_H(\theta)| (E[X_{k}^{2}]+ f_H(\theta)) \\
& \leqslant& \frac{C}{n} \sum_{k =1}^{n} |E[X_{k}^{2}] - f_H(\theta)|
\leqslant \frac{C}{n} \sum_{k=1}^{n} k^{2H-2} \\
& \leqslant& C \varphi_{H}(n).
\end{eqnarray*}
\end{proof}

\begin{lemma}
Let $X$ be the solution of (\ref{subFOU}). If $0 < H< 3/4$, we have
\begin{eqnarray*}
J(n):= \left|\frac{1}{n} \sum \limits_{\underset{ k \neq j }{j,k=1}}^n
(E[X_{k}X_{j}])^{2} - 2 \sum_{i \in \mathbb{N}^{*}} \rho_H(i)^{2}\right|
\leqslant C \psi_{H}(n).
\end{eqnarray*}
\end{lemma}

\begin{proof}
Suppose that $H \neq 1/2$. Using (\ref{cov of X GOU}), we have
\begin{eqnarray*}
\frac{1}{n} \sum_{k,j =1}^{n} (E[X_{k}X_{j}])^{2} & = E_{1}(n)+ E_{2}(n)+
E_{3}(n)+ E_{4}(n)+ E_{5}(n)+ E_{6}(n)
\end{eqnarray*}
where
\begin{equation*}
E_{1}(n) := \frac{2}{n} \sum_{j =1}^{n-1} \sum_{k = j+1}^{n}
a(k,j)^{2},\quad E_{2}(n) := \frac{4}{n}\sum_{j =1}^{n-1} \sum_{k = j+1}^{n}
a(k,j) A(k,j)
\end{equation*}

\begin{equation*}
E_{3}(n) := \frac{4}{n} \sum_{j =1}^{n-1} \sum_{k = j+1}^{n} a(k,j)
D(k,j),\quad E_{4}(n) := \frac{2}{n} \sum_{j =1}^{n-1} \sum_{k = j+1}^{n}
A(k,j)^{2}
\end{equation*}
\begin{equation*}
E_{5}(n) := \frac{4}{n} \sum_{j =1}^{n-1} \sum_{k = j+1}^{n}
A(k,j)D(k,j),\quad E_{6}(n) := \frac{2}{n} \sum_{j =1}^{n-1} \sum_{k =
j+1}^{n} D(k,j)^{2}
\end{equation*}
with for every $s\leqslant t$
\begin{equation*}
a(s,t) := \e^{- \theta (t-s)} E[X_{s}^{2}],
\end{equation*}
\begin{equation*}
A(t,s) := H(2H-1) \e^{- \theta t} \e^{- \theta s} \int_{s}^{t} \e^{\theta v}
\int_{0}^{s} \e^{\theta u} (u-v)^{2H-2} du dv,
\end{equation*}
\begin{equation*}
D(t,s) := -H(2H-1) \e^{- \theta t} \e^{- \theta s} \int_{s}^{t} \e^{\theta
v} \int_{0}^{s} \e^{\theta u} (u+v)^{2H-2} du dv.
\end{equation*}
Then, we can write
\begin{eqnarray*}
J(n) &\leqslant& |E_{1}(n) - \frac{2 f_H(\theta)^{2}\e^{-2 \theta}}{1- \e%
^{-2 \theta}}| + \left|E_{2}(n) - 4 f_H(\theta)\left( \sum_{i =1}^{\infty} \e%
^{- \theta i} \rho(i) - \frac{ f_H(\theta)e^{-2 \theta}}{1- \e^{- 2\theta}}%
\right)\right| + |E_{3}(n)| \\
&& + \left|E_{4}(n) - 2 \sum_{i =1}^{\infty} \left(\rho(i)-f_H(\theta)\e^{-
\theta i}\right)^{2}\right|+ |E_{5}(n)| + |E_{6}(n)|.
\end{eqnarray*}
On the other hand,
\begin{eqnarray}
\left|E_{1}(n) - \frac{2 f_H(\theta)^{2}\e^{-2 \theta}}{1- \e^{-2 \theta}}%
\right| &\leqslant& \frac{2}{n} \sum_{j=1}^{n-1} \sum_{k = j+1}^{n} \e^{- 2
\theta (k-j)} \left([E(X_j^2)]^{2}- f_H(\theta)^{2}\right)  \notag \\
&& + \left|\frac{2}{n} \sum_{j= 1}^{n-1}\sum_{k= j+1}^{n} \e^{- 2
\theta(k-j)} f_H(\theta)^{2} - \frac{2 f_H(\theta)^{2}\e^{-2 \theta}}{1- \e%
^{-2 \theta}}\right|  \notag \\
&\leqslant&\frac{2 C}{n} \sum_{j=1}^{n-1} \sum_{k = j+1}^{n} \e^{- 2 \theta
(k-j)} j^{2H-2} +2 f_H(\theta)^{2} \left(\sum_{i=n}^{\infty} e^{- 2 \theta
i} + \frac{1}{n} \sum_{i=1}^{n-1} i \e^{- 2 \theta i}\right)  \notag \\
& \leqslant& C\left[\frac{1}{n} \sum_{j=1}^{n-1} j^{2H-2} \sum _{i=1}^{n-j}
e^{- 2 \theta i}+ e^{- 2 \theta n} + n^{-1}\right]  \notag \\
&\leqslant& C \varphi_{H}(n).  \label{E_1}
\end{eqnarray}
Furthermore,
\begin{eqnarray*}
E_{2}(n) & = & \frac{4}{n} \sum_{j=1}^{n-1} \sum_{k = j+1}^{n} \e^{- \theta
(k-j)} E(X_j^2) [ E[Z_{k-j}Z_{0}] - \e^{- \theta(k-j)} E[Z_{0}^{2}] - \e^{-
\theta j} E[Z_{0}Z_{k}] +\e^{- \theta k} E[Z_{0}Z_{j}]] \\
& =& \frac{4}{n} \sum_{j=1}^{n-1} \sum_{k = j+1}^{n} \e^{- \theta (k-j)}
E(X_j^2) \rho(k-j) - \frac{4 \rho(0)}{n} \sum_{j=1}^{n-1} \sum_{k = j+1}^{n} %
\e^{- 2 \theta(k-j)} E(X_j^2) \\
& & - \frac{4}{n} \sum_{j=1}^{n-1} \e^{- \theta j} E(X_j^2) \sum_{k =
j+1}^{n} \e^{- \theta(k-j)} \rho(k)+ \frac{4}{n} \sum_{j=1}^{n-1}
\rho(j)E(X_j^2) \sum_{k = j+1}^{n} \e^{- \theta(k-j)} \e^{- \theta k} \\
&:=&E_{2}^{1}(n)-E_{2}^{2}(n)-E_{2}^{3}(n)+E_{2}^{4}(n).
\end{eqnarray*}
For $E_{2}^{1}(n)$, we have
\begin{eqnarray*}
\left|E_{2}^{1}(n)- 4 f_H(\theta)\sum_{i \in \mathbb{N}^{*}} \e^{- \theta i}
\rho(i)\right| &\leqslant& \frac{4}{n} \sum_{j=1}^{n-1} \sum_{k = j+1}^{n} \e%
^{- \theta (k-j)} |E(X_j^2)- f_H(\theta)| |\rho(k-j)| \\
&& + 4|f_H(\theta)| \left|\frac{1}{n}\sum_{j=1}^{n-1} \sum_{k = j+1}^{n} \e%
^{- \theta (k-j)}\rho(k-j) - \sum_{i=1}^{\infty} \e^{- \theta i}
\rho(i)\right| \\
&\leqslant& C\left(n^{2H-2}+n^{-1}\right)
\end{eqnarray*}
because
\begin{eqnarray*}
\left|\frac{1}{n}\sum_{j=1}^{n-1} \sum_{k = j+1}^{n} \e^{- \theta
(k-j)}\rho(k-j) - \sum_{i=1}^{\infty} \e^{- \theta i} \rho(i)\right| & =&
\left|\sum_{i=n}^{\infty} \e^{- \theta i} \rho(i) - \frac{1}{n} \sum_{i =
1}^{n-1} i\e^{- \theta i} \rho(i)\right| \\
& \leqslant & \sum_{i= n}^{\infty} \e^{- \theta i}|\rho(i)| + \frac{1}{n}
\sum_{i = 1}^{n-1} i |\rho(i)|\e^{- \theta i} \\
& \leqslant & C\left(n^{2H-2}+n^{-1}\right)
\end{eqnarray*}
and by using a similar argument as in above, we also have
\begin{eqnarray*}
\frac{4}{n} \sum_{j=1}^{n-1} \sum_{k = j+1}^{n} \e^{- \theta (k-j)}
|E(X_j^2)- f_H(\theta)| |\rho(k-j)| &\leqslant&
C\left(n^{2H-2}+n^{-1}\right).
\end{eqnarray*}
By straightforward calculus, we also have
\begin{eqnarray*}
|E_{2}^{2}(n)- 4 f_H(\theta)\sum_{i =1}^{\infty} \e^{-2 \theta i} |
\leqslant C \varphi_{H}(n)
\end{eqnarray*}
For $E_{2}^{3}(n)$, we have
\begin{eqnarray*}
|E_{2}^{3}(n)| & \leqslant& \frac{4}{n} \sum_{j=1}^{n-1} \e^{- \theta j}
|E(X_j^2)-f_H(\theta)| \sum_{k = j+1}^{n} \e^{- \theta(k-j)} |\rho(k)| +
\frac{4 |f_H(\theta)|}{n} \sum_{j=1}^{n-1} \e^{- \theta j}\sum_{k = j+1}^{n} %
\e^{- \theta(k-j)} |\rho(k)| \\
& \leqslant& C n^{-1}.
\end{eqnarray*}
For $E_{2}^{4}(n)$, since
\begin{eqnarray*}
\frac{4}{n} \sum_{j =1}^{n-1} \sum_{k = j+1}^{n} \e^{- \theta(k-j)}
|\rho(j)||E(X_j^2)-f_H(\theta)| \e^{- \theta k} \leqslant \frac{C}{n}
\sum_{j = 1}^{n-1} j^{4H-4} \leqslant \frac{C}{n},
\end{eqnarray*}
and
\begin{eqnarray*}
\frac{4}{n}\sum_{j=1}^{n-1} |\rho(j)| \sum_{k = j+1}^{n} \e^{- \theta(k-j)} %
\e^{- \theta k} \leqslant \frac{C}{n} \sum_{j =1}^{n-1} j^{2H-2} \leqslant C
\varphi_{H}(n)
\end{eqnarray*}
we get
\begin{eqnarray*}
|E_{2}^{4}(n)| \leqslant C \varphi_{H}(n).
\end{eqnarray*}
Thus, we conclude
\begin{eqnarray}
\left|E_{2}(n) - 4 f_H(\theta)\left( \sum_{i =1}^{\infty} \e^{- \theta i}
\rho(i) - \frac{ f_H(\theta)e^{-2 \theta}}{1- \e^{- 2\theta}}\right)\right|
\leqslant C \varphi_{H}(n).  \label{E_2}
\end{eqnarray}
For $E_{3}(n)$, since for every $\forall k > j$, $|D(k,j)| \leqslant C
(jk)^{H-1}$, we can write
\begin{eqnarray}
|E_{3}(n)| & \leqslant& \frac{C}{n} \sum_{j=1}^{n-1} \sum_{k = j+1}^{n} \e%
^{- \theta (k-j)} |E(X_j^2)| (kj)^{H-1}  \notag \\
& \leqslant& \frac{C}{n}\left[\sum_{j=1}^{n-1} \sum_{k = j+1}^{n} \e^{-
\theta (k-j)} |E(X_j^2)- f_H(\theta)| (kj)^{H-1} + f_H(\theta)
\sum_{j=1}^{n-1} \sum_{k = j+1}^{n} \e^{- \theta (k-j)} (kj)^{H-1}\right]
\notag \\
& \leqslant& \frac{C}{n}\left[\sum_{j =1}^{n} j^{4H-4}+\sum_{j = 1}^{n-1}
j^{2H-2}\right]  \notag \\
& \leqslant& C\varphi_{H}(n).  \label{E_3}
\end{eqnarray}
For $E_{4}(n)$, we first calculate $A(k,j)^{2}$. We have
\begin{eqnarray*}
A(k,j)^{2} & =& \rho(k-j)^{2} - 2 \rho(k-j)\e^{- \theta (k-j)} \rho(0) +
\rho(0)^{2} \e^{- 2 \theta (k-j)} + \e^{- 2 \theta k} \rho(j)^{2} \\
&& - 2 \e^{- \theta(k+j)} \rho(j) \rho(k) + \e^{- 2 \theta j} \rho(k)^{2} +
2 \e^{- \theta k} \rho(k-j)\rho(j) - 2 \e^{- \theta j} \rho(k) \rho(k-j) \\
&&- 2 \e^{- \theta (k-j)} \rho(0) \e^{- \theta k} \rho(j) + 2 \e^{-
\theta(k-j)} \rho(0) e^{- \theta j} \rho(k).
\end{eqnarray*}
Hence, we need to study the rate of convergence of the following terms. We
have
\begin{eqnarray*}
\left|\frac{2}{n} \sum_{j = 1}^{n-1} \sum_{k = j+1}^{n} \rho(k-j)^{2} - 2
\sum_{i =1}^{\infty} \rho(i)^{2} \right| &\leqslant& 2 \sum_{i = n}^{\infty}
i^{4H-4} + \frac{2}{n} \sum_{i = 1}^{n -1}i^{4H-3} \\
&\leqslant& C \psi_{H}(n),
\end{eqnarray*}
\begin{eqnarray*}
4 \rho(0) \left|\frac{1}{n} \sum_{j= 1}^{n -1} \sum_{k = j+1}^{n} \e^{-
\theta(k-j)} \rho(k-j) - \sum_{i =1}^{\infty} \e^{- \theta i} \rho(i)\right|
&\leqslant& 4 \rho(0) \sum_{j =n}^{\infty} j^{2H-2} \e^{- \theta j} \\
&& + \frac{4 \rho(0)}{n} \sum_{ i = 1}^{n-1} \e^{- \theta i} i^{2H-1} \\
&\leqslant& C \varphi_{H}(n)
\end{eqnarray*}
and
\begin{eqnarray*}
2 \rho(0)^{2} \left|\frac{1}{n} \sum_{j =1}^{n-1}\sum_{k=j+1}^{n} \e^{-
2\theta(k-j)} - \sum_{i =1}^{\infty} \e^{- 2 \theta i}\right| \leqslant C
n^{-1}.
\end{eqnarray*}
Moreover, it is clear that
\begin{eqnarray*}
\frac{1}{n}\sum_{j = 1}^{n-1} \sum_{k = j+1}^{n}\left[ \e^{- 2 \theta k}
\rho(j)^{2}+\e^{- \theta (k+j)} |\rho(j)| |\rho(k)|+\e^{- 2 \theta j}
\rho(k)^{2}\right] \leqslant C n^{-1}
\end{eqnarray*}
and
\begin{eqnarray*}
\frac{1}{n} \sum_{j =1}^{n-1} \sum_{k = j+1}^{n} \e^{- \theta k} |\rho(k-j)|
|\rho(j)| \leqslant \frac{C }{n} \sum_{j=1}^{n-1} j^{2H-2} \leqslant C
\varphi_{H}(n).
\end{eqnarray*}
In addition,
\begin{eqnarray*}
\frac{4}{n} \sum_{j =1}^{n-1} \sum_{k = j+1}^{n} \e^{- \theta j} |\rho(k)|
|\rho(k-j)| & \leqslant& \frac{C}{n} \sum_{j =1}^{n-1} \e^{- \theta j}
\sum_{k = j+1}^{n} (k-j)^{4H-4} \\
& \leqslant& \frac{C}{n} \sum_{i =1}^{\infty} i^{4H-4} \sum_{j=1}^{n-1} \e%
^{- \theta j} \leqslant C n^{-1},
\end{eqnarray*}
\begin{eqnarray*}
\frac{4}{n} \sum_{j =1}^{n-1} \sum_{k = j+1}^{n} \e^{- \theta (k-j)}
|\rho(j)| \e^{- \theta k} & \leqslant& \frac{4 }{n} \sum_{j=1}^{n-1}
j^{2H-2} \sum_{k= j+1}^{n} \e^{- \theta(k-j)} \\
& \leqslant& C \varphi_{H}(n),
\end{eqnarray*}
\begin{eqnarray*}
\frac{4 \rho(0)}{n} \sum_{j =1}^{n-1} \sum_{k = j+1}^{n} \e^{-\theta(k-j)}
|\rho(k)| \e^{- \theta j} & \leqslant & \frac{4 \rho(0)}{n} \sum_{j =
1}^{n-1} \e^{- \theta j} \sum_{k = j+1}^{n} \e^{- \theta(k-j)} (k-j)^{2H-2}
\\
& \leqslant & \frac{4 \rho(0)}{n} \sum_{j=1}^{n-1} e^{- \theta j} \sum_{i
=1}^{\infty} \e^{- \theta i} \rho(i) \\
& \leqslant& C n^{-1}.
\end{eqnarray*}
Thus,
\begin{eqnarray}
\left|E_{4}(n) - 2 \sum_{i =1}^{\infty} \left(\rho(i)-f_H(\theta)\e^{-
\theta i}\right)^{2}\right| \leqslant C \psi_{H}(n).  \label{E_4}
\end{eqnarray}

For $E_{5}(n)$, we have
\begin{eqnarray}
|E_{5}(n)| &\leqslant& \frac{C}{n} \sum_{j = 1}^{n-1} \sum_{k = j+1}^{n}
(k-j)^{2H-2} (kj)^{H-1}  \notag \\
&\leqslant& \frac{C}{n} \sum_ {j =1}^{n-1} j^{H-1} \sum_{i =1}^{n-j}
i^{2H-2} (i+j)^{H-1}  \notag \\
&\leqslant& \frac{C}{n} \left(\sum_{j=1}^{n} j^{2H-2}\right)^{2} \leqslant C
\psi_{H}(n).  \label{E_5}
\end{eqnarray}
Finally, using $D(k,j)^{2} \leqslant |A(k,j)||D(k,j)|$ $\forall k > j$, we
get
\begin{eqnarray}
|E_{6}(n)| \leqslant C |E_{5}(n)|.  \label{E_6}
\end{eqnarray}
Combining (\ref{E_1}), (\ref{E_2}), (\ref{E_3}), (\ref{E_4}), (\ref{E_5})
and (\ref{E_6}) the proof is completed.
\end{proof}

Let us now study the case when $H = 3/4$.

\begin{lemma}
Let $X$ be the solution of (\ref{subFOU}). If $H = \frac{3}{4}$, we have
\begin{eqnarray}
\left|\frac{1}{n \log(n)} \sum \limits_{\underset{ k \neq j }{j,k=1}}^n
\left(E[X_{k}X_{j}]\right)^{2} - \frac{9}{32\theta^{4}}\right| \leqslant
\frac{C}{\log(n)} .
\end{eqnarray}
\end{lemma}

\begin{proof}
Using similar argument as in above, it is straightforward to check that
\begin{eqnarray*}
|E_{1}(n)| +|E_{2}(n)| + |E_{3}(n)|+ |E_{5}(n)|+ |E_{6}(n)|\leqslant C,
\end{eqnarray*}
and
\begin{eqnarray*}
\left|\frac{1}{\log(n)}E_{4}(n)-\frac{9}{32 \theta^{4}}\right| \leqslant C
\log(n)^{-1}
\end{eqnarray*}
which completes the proof.
\end{proof}

\begin{lemma}
\label{variance case bifBm}Let $X$ be the solution of (\ref{GOU}), with the
process $G$ is a bifBm. If $0 < HK < 3/4 $, we have
\begin{eqnarray*}
\left| E[V_{n}(X)^{2}] - 4 \sum_{i =1}^{\infty} \left[\rho(i) - (1-2^{1-K})\e%
^{-\theta i} \rho(0)\right]^{2} - 2^{3-2K} \rho(0)^{2}\right| \leqslant C
\psi_{HK}(n).
\end{eqnarray*}
If $HK = 3/4$, we have
\begin{eqnarray*}
\left|\frac{E[V_{n}(X)^{2}]}{\log(n)}- \frac{9}{16 \theta^{4}}\right|
\leqslant C \log(n)^{-1}.
\end{eqnarray*}
\end{lemma}

\begin{proof}
We can write
\begin{equation*}
E[V_{n}(X)^{2}] = \frac{2}{n} \sum \limits_{\underset{ k \neq l }{k,l=1}}^n
(E[X_{k}X_{j}])^{2} + \frac{2}{n} \sum_{k =1}^{n} (E[X_{k}^{2}])^{2},
\end{equation*}
where
\begin{eqnarray*}
|\frac{1}{n}\sum_{k=1}^{n}(E[X_{k}^{2}])^{2}-2^{2-2K}\rho (0)^{2}|
&\leqslant &\frac{1}{n}\sum_{k=1}^{n}|E[X_{k}^{2}]-2^{1-K}\rho
(0)|(E[X_{k}^{2}]+2^{1-K}\rho (0)) \\
&\leqslant &\frac{C}{n}\sum_{k=1}^{n}|E[X_{k}^{2}]-2^{1-K}\rho (0)| \\
&\leqslant &\frac{C}{n}\sum_{k=1}^{n}k^{2HK-2} \\
&\leqslant &C\varphi _{HK}(n).
\end{eqnarray*}%
Moreover,
\begin{equation*}
\left\vert \frac{1}{n}\sum\limits_{\underset{k\neq j}{j,k=1}%
}^{n}(E[X_{k}X_{j}])^{2}-2\sum_{i=1}^{\infty }(\rho (i)-(1-2^{1-K})\e%
^{-\theta i}\rho (0))^{2}\right\vert \leqslant C\psi _{HK}(n).
\end{equation*}%
Indeed, we have
\begin{equation*}
\frac{1}{n}\sum_{k,j=1}^{n}(E[X_{k}X_{j}])^{2}=\overline{E}_{1}(n)+\overline{%
E}_{2}(n)+\overline{E}_{3}(n)+\overline{E}_{4}(n)+\overline{E}_{5}(n)+%
\overline{E}_{6}(n)
\end{equation*}%
where
\begin{equation*}
\overline{E}_{1}(n):=\frac{2}{n}\sum_{j=1}^{n-1}\sum_{k=j+1}^{n}a(k,j)^{2},%
\quad \overline{E}_{2}(n):=\frac{4}{n}\sum_{j=1}^{n-1}\sum_{k=j+1}^{n}a(k,j)%
\overline{A}(k,j)
\end{equation*}%
\begin{equation*}
\overline{E}_{3}(n):=\frac{4}{n}\sum_{j=1}^{n-1}\sum_{k=j+1}^{n}a(k,j)%
\overline{D}(k,j),\quad \overline{E}_{4}(n):=\frac{2}{n}\sum_{j=1}^{n-1}%
\sum_{k=j+1}^{n}\overline{A}(k,j)^{2}
\end{equation*}%
\begin{equation*}
\overline{E}_{5}(n):=\frac{4}{n}\sum_{j=1}^{n-1}\sum_{k=j+1}^{n}\overline{A}%
(k,j)\overline{D}(k,j),\quad \overline{E}_{6}(n):=\frac{2}{n}%
\sum_{j=1}^{n-1}\sum_{k=j+1}^{n}\overline{D}(k,j)^{2}
\end{equation*}%
with for $0<s<t$,
\begin{eqnarray*}
a(s,t):= &&\e^{-\theta (t-s)}E[X_{s}^{2}], \\
\overline{A}(t,s):= &&2^{1-K}HK(2HK-1)\e^{-\theta t}\e^{-\theta
s}\int_{s}^{t}\e^{\theta v}\int_{0}^{s}\e^{\theta u}(u-v)^{2HK-2}dudv, \\
\overline{D}(t,s):= &&2^{2-K}K(K-1)\e^{-\theta t}\e^{-\theta
s}\int_{s}^{t}\int_{0}^{s}\e^{\theta u}\e^{\theta
v}(uv)^{2H-1}(u^{2H}+v^{2H})^{K-2}dudv.
\end{eqnarray*}%
Using similar arguments as in the subfractional case, we obtain
\begin{eqnarray*}
&&\left\vert \overline{E}_{1}(n)-\frac{2^{3-2K}\rho (0)^{2}\e^{-2\theta }}{1-%
\e^{-2\theta }}\right\vert +\left\vert \overline{E}_{2}(n)-2^{3-K}\rho
(0)\sum_{i=1}^{\infty }\e^{-\theta i}\rho (i)-2^{3-K}\rho (0)^{2}\frac{\e%
^{-2\theta }}{1-\e^{-2\theta }}\right\vert \\
&&+\left\vert \overline{E}_{3}(n)\right\vert +\left\vert \overline{E}%
_{4}(n)-2\sum_{i=1}^{\infty }(\rho (i)-\rho (0)\e^{-\theta
i})^{2}\right\vert +\left\vert \overline{E}_{5}(n)\right\vert +\left\vert
\overline{E}_{6}(n)\right\vert \\
&\leqslant &C\psi _{HK}(n)
\end{eqnarray*}%
which completes the proof of the first inequality of Lemma \ref{variance
case bifBm}.\newline
By using similar techniques as in above we also obtain for $HK=3/4$,
\begin{equation*}
\left\vert \frac{1}{n\log (n)}\sum\limits_{\underset{k\neq j}{j,k=1}%
}^{n}(E[X_{k}X_{j}])^{2}-\frac{9}{32}\frac{1}{\theta ^{4}}\right\vert
\leqslant C\log (n)^{-1}
\end{equation*}%
which finishes the proof.
\end{proof}

\end{document}